\documentclass[reqno, makeidx]{amsart}
\usepackage{amsmath}
\usepackage{latexsym}
\usepackage{amsfonts}
\usepackage{amssymb}
\usepackage{amsthm}
\usepackage{amsbsy}
\usepackage{graphicx}
\usepackage{esint}
\usepackage{epsf}
\usepackage{color}
\usepackage{epstopdf}
\usepackage[all]{xy}

\newcommand{\R}{\mathbb{R}}
\newcommand{\N}{\mathbb{N}}
\newcommand{\G}{\mathbb{G}}
\newcommand{\C}{\mathbb{C}}

\newcommand{\dH}{\mathbb{H}}
\newcommand{\dP}{\mathbb{P}}
\newcommand{\dE}{\mathbb{E}}
\newcommand{\Dim}{\mathrm{dim}}
\newcommand{\Span}{\mathrm{Span}}

\newcommand{\ess}{\mathrm{ess}}

\newcommand{\cX}{\mathcal{X}}

\newcommand{\cL}{\mathcal{L}}
\newcommand{\cA}{\mathcal{A}}

\newcommand{\cv}{\mathcal{\nu}}
\newcommand{\cV}{\mathcal{V}}

\newcommand{\cH}{\mathcal{H}}

\newcommand{\tr}{\textrm{Tr}}
\newcommand{\norma}[1]{\left \|#1\right \|}

\newtheorem{defi}{Definition}[section]
\newtheorem{lemma}{Lemma}[section]
\newtheorem{teo}{Theorem}[section]
\newtheorem{coroll}{Corollary}[section]
\newtheorem{rem}{Remark}[section]
\newtheorem{prop}{Proposition}[section]
\newtheorem{ex}{Example}[section]

\begin{document}
\title[evolution by mean curv. flow in sub-Riem. geom.: a stoch. approach]{Evolution by mean curvature flow in sub-Riemannian geometries:  a stochastic approach.}
\author{Nicolas Dirr, Federica Dragoni, Max von Renesse}
\date{}
\begin{abstract} 
We study the phenomenon of  evolution by horizontal mean curvature flow in sub-Riemannian geometries. We use a stochastic approach to 
prove the existence of a generalized evolution in these spaces. In particular we show that the value function  of suitable family of stochastic control problems  solves in the viscosity sense the level set equation for the evolution by horizontal mean curvature flow.
\end{abstract}
\maketitle
\section{Introduction.}
In Euclidean spaces, the motion by mean curvature flow  of a hypersurface is a geometrical evolution such that the normal velocity at each point of the hypersurface is
 equal to  mean curvature at that point. 
 Unfortunately, even  smooth surfaces  can develop singularities in finite time, so a weak notion  of evolution is necessary.\\
  A well-known  example in $\R^3$ is the surface given by two huge spheres, smoothly connected by a long straight cylinder. When that surface  evolves by mean curvature,  it splits, in finite time,  in two connected components which are topological spheres.  
At the moment when this happens, the surface intersects itself. In the points of the self-intersection, the normal and so the mean curvature   are not defined (e.g. \cite{dumbbell}).\\
 Different notions of generalized evolutions have been introduced in order to study the evolution of surfaces beyond the formation of singularities. 
Brakke defined in \cite{Brakke} a varifold-based concept of weak solution which provides existence but no uniqueness, while  a variational approach was developed by Almgren, Taylor and Wang~\cite{ATW:93} and Luckhaus and  Sturzenhecker~\cite{LS:95}.
Another approach, the so-called barrier solutions, has been introduced by E. De Giorgi~\cite{DG:94} and developed by  G. Bellettini, M. Paolini and M. Novaga \cite{{BePa:95},{BN:97}}, while  a more recent approach was found by G. Barles and P.E. Souganidis, \cite{souganidis}.\\
The notion that we are going to use follows  a nonlinear-PDE-approach, found for the first time in 1991,  independently, by 
Chen-Giga-Goto (\cite{Giga}) and 
Evans- Spruck (\cite{EvSp:91}). Roughly speaking, the idea consists in associating a PDE to a smooth hypersurface evolving by mean curvature flow such 
that  the function (of space and time) which solves this PDE has level sets which evolve by mean curvature flow. Then one can define the  solutions of the ``generalized evolution by mean curvature flow'' as the hypersurface given by the zero-level sets of the viscosity solution of this PDE.
 In this paper we  study the corresponding evolution in sub-Riemannian geometries with the help of stochastic control methods.\\
 Sub-Riemannian geometries are degenerate Riemannian spaces where the Riemannian inner product is defined just on a sub-bundle of the tangent bundle. To be more precise, we will consider $X_1,...,X_m$ smooth vector fields on $\R^n$  and a Riemannian inner product defined on the distribution $\cH$ generated by such vector fields. We assume that the vector fields verify the H\"ormander condition which means that the associated Lie algebra is equal to $\R^n$ at any point. This condition has many important consequences. The main one is that we can always connect  two points by an admissible path. i.e. an absolutely continuous curve such that the velocity belongs to the distribution for almost any time (Chow's Theorem). Therefore we can always define an associated distance $d(x,y)$ on the whole $\R^n$.
 These spaces are topologically equivalent to the Euclidean $\R^n$ but the metric is not. In fact, instead of the equivalence, we have only the following inequality
 $c|x-y|\leq d(x,y)\leq C|x-y|^{\frac{1}{k}}$, for some $c,C>0$ constants and $k>1$. 
 For some recent applications of these geometries we refer to the papers by Citti and Sarti on the visual cortex (\cite{citti}).
In a  sub-Riemannian space
 it is possible to define intrinsic derivatives of any order taking the derivatives along the vector fields $X_1,...,X_m$. 
 That allows us to write differential operators like Laplacian, infinite-Laplacian etc, using intrinsic derivatives. We call these operators ``horizontal''. 
In particular we can define a notion of horizontal mean curvature flow.\\
 While there are many results for evolution by mean curvature flow in  the  Euclidean setting,
 only little is known in these degenerate spaces.
This evolution in a sub-Riemannian manifold 
is very different from the corresponding Euclidean motion, in particular because of the existence of the  so-called characteristic points,
which are points where the Euclidean normal is perpendicular to the horizontal space and so not admissible. 
 Even if their geometrical meaning is very different, at the level of the level-set PDE, they look almost like the  Euclidean "singularities," 
i.e. like the points where the level-set function has vanishing gradient. While in the Euclidean case the only  
"problematic" value of the gradient is zero, the corresponding set of gradients in our case is space-dependent and 
has, at each point,  nonzero dimension. 
The different nature of these degeneracies creates serious difficulties in applying most of those techniques which are known to
work for the Euclidean setting.
To avoid the problems created from the presence of these singularities, we will use a stochastic approach for showing  existence of solutions.\\

A connection between a  certain stochastic control problem 
and a  large class of geometric evolution equations, including 
the (Euclidean) evolution by mean curvature flow,  has been found by
Buckdahn,  Cardaliaguet and Quincampoix 
(in \cite{buckdahn}) and  Soner and Touzi (in \cite{{soner3},{soner1}}).
  The control, loosely speaking, constrains  the increments of the stochastic process
to a lower dimensional subspace of $\R^n,$ while the cost functional consists only of the terminal cost but
involves an essential supremum over the probability space.
It turns out that the value function solves
the level set equation associated with the geometric evolution. Moreover, one can show that the set of points from which the initial hypersurface can be reached
almost surely in a given time by choosing an appropriate control coincides with the set evolving by mean curvature flow. 
This stochastic approach generalizes very naturally to sub-Riemannian geometries.
Instead of {constraining} the Euclidean Brownian motion, we use an intrinsic Brownian motion associated with the 
sub-Riemannian geometry. This allows us to obtain certain existence results in general sub-Riemannian manifolds which 
were previously unknown.\\
Our construction of controlled paths  yields an analogue to the processes considered for the Euclidean case in \cite{buckdahn,soner3,soner1}, which 
could be called \textit{locally codimension one constrained Brownian motion}. In the Euclidean case, any control $\cv(s)$, 
taking values in the space of co-rank-one orthogonal projections induces a \emph{locally codimension one constrained  Brownian motion $B^\cv$} as solution of the following It\^o SDE $dB^{\cv} = \cv(s) dB$. 
In the present sub-Riemannian case, in order to define the  locally codimension one constrained  or unconstrained   Brownian motion, some extra care has to be taken  due to the  geometry. 
 We define an  ``horizontal Brownian motion''  as the stochastic process whome  generator is the horizontal Laplacian operator $\Delta _0 = \sum_{i=1}^m X_i^2 .$ 
 We  would  like to remark that, unlike the Euclidean case, the  horizontal  Laplacian has in general also a first order part, coming from 
taking the derivatives of the vector fields. 
The construction of the associated \emph{unconstrained horizontal Brownian motion}  by means of the following Stratonovich SDE, is natural:
$d\xi(s) = \sum_{i=1}^m X_i(\xi(s)) \circ dB^{i}(s),$ 
 where  $B=(B^1, \dots, B^m)$ is a standard Brownian motion in $\R^m$.\\
Replacing in previous Stratonovich SDE  the unconstrained Brownian motion  $B$  by a 
locally codimension one constrained  Euclidean Brownian motion $B^\cv$ in $\R^m$,  we get the \emph{locally constrained codimension one horizontal Brownian motion} $\xi^{\cv}(s)$ associated to $\Delta_0$ and $\cv(s)$, which  constitutes a  controlled path for  our problem.\\
 Note that, in the sub-Riemannian case, the process defined by the Stratonovich 
SDE above differs from the one defined by the corresponding It\^o integral just in the ``vertical direction'' (a consequence of \cite{garofalo}), so that their horizontal projections coincide. Moreover, in the Heisenberg group, it is possible to show, by an explicit calculation, that they are the same also in the ``vertical'' direction.\\
  The value function associated to this stochastic control problem is defined as the  infimum, over the admissible controls, of the 
essential   supremum of the final cost $g$ (at some fixed terminal time $T>t$),   for the controlled path $\xi^{\cv}$  starting from $x$ at the time $t$. 
We can show that this value function is a representation formula for the generalized evolution by horizontal mean curvature flow, i.e.  solves in the (discontinuous) viscosity sense the equation 
$-v_t+H(x,Dv,D^2v)=0$ where $H(x,Dv,D^2v)=\Delta_0 v-\Delta_{0,\infty}$, with terminal condition $g$.
That means that $u(t,x ):=v(T-t,x)$ is a viscosity solution of $u_t+H(x,Du,D^2u)=0$ with initial condition $g$, which is exactly the level set equation of the  evolution by horizontal mean curvature flow. 
So $\Gamma(t)=\{x\in \R^n\,|\, u(t,x)=0\}$ is a 
(discontinuous) generalized evolution by horizontal mean curvature flow in general sub-Riemannian  manifolds. Whenever comparison principles  hold, we are able to show that such an evolution is also continuous. Unfortunately, comparison principles  are still 
an open problem in sub-Riemannian geometries. As far as we know, there exists  just  a recent preprint by Capogna and Citti (\cite{capognaCitti}) where the authors obtain  comparison principles in Carnot groups but only for  a  particular class of initial data.
In the same paper, the authors show also the existence, using a definition of generalized motion by hoizontal mean curvature flow slightly different from our notion (but equivalent at least in the Euclidean case). Their approach is very different from ours. In fact, they get the solution not by a representation formula but as limit of solutions of suitable approximating non-degenerate parabolic PDEs.
\\
The organization of the paper is the following.\\

In Section 2 we introduce  sub-Riemannian geometries. We recall the definitions of the H\"ormander condition and the special case of Carnot groups. The main example for a Carnot group is the Heisenberg group hence, from time to time,  we will  focus on such a particular geometry.  Then we give the definition of horizontal mean curvature and of the characteristic points.  Note that,
there are very   few  interesting
hypersurfaces  without characteristic points
 (for example in the Heisenberg group, any compact surface topologically equivalent to a sphere, has at least one characteristic point). Therefore existence results for short times with ``smooth'' initial data are not as helpful as in the Euclidean case.\\
 We conclude the section looking closer at the case of the Heisenberg group and  giving several explicit examples in this geometry.\\

In Section  3 we introduce a notion of generalized evolution by mean curvature flow, following the level set formulation introduced by Chen-Giga-Goto in \cite{Giga} for the corresponding Euclidean evolution.  So we associate to the evolving surface a degenerate parabolic equation, where the degeneracy arises   where the horizontal gradient vanishes, as it happens  at the characteristic points.  At these points, the differential operator as a function
of first and second derivatives at the point, is not continuous
any more. In the Euclidean case this situation arises only 
when the gradient is zero,  i.e. a set of dimension zero
in $\R^{n\times n}\times \R^n.$ In the sub-Riemannian case, however,
this set is both of non-zero dimension and depend on space, which
makes the analysis far more difficult.\\

 In Section  4 we define   and study a stochastic control problem,   whose
associated value function   solves  in the viscosity sense the level set equation for the evolution 
by horizontal mean curvature flow, introduced in section 3. We introduce a family of   (Stratonovich) stochastic ODEs driven by a 
``horizontal constrained Browinian motion'' and we will show that the associated generator is exactly the horizontal Laplacian. 
 Moreover we study  some properties of value function. 
In particular, we show a property which implies that the time evolution of its zero level set depends on the terminal value
only through the zero level set, and, moreover, that 
in Carnot groups the value function is bounded and continuous in space whenever the terminal cost is bounded and uniformly continuous and 
has a  a  limit as $|x|\to \infty$.\\

 In Section 5   we show that the value function is a bounded and lower semicontinuous viscosity solution of the level set equation for the  evolution by horizontal mean curvature flow in the sub-Riemannian case.  We first derive the PDE solved by the value function, 
assuming more regularities for the solution. This proof  gives also a justification on why the optimal control is, at any point, the projection on the horizontal tangent space of the level set. 
 We give several examples of sub-Riemannian geometries covered by our existence result. We conclude dealing  with the continuity of the  stochastic representation formula found. 
We show that if there exist comparison principles for the degenerate parabolic PDE introduced  in section 3,  the value function is   continuous in any sub-Riemannian geometry and, therefore,  it is a classic viscosity solution in the sense of Crandall and Lions. Nevertheless, as we already remarked, comparisons are known just in very few cases.\\

 In the Appendix we give some of the technical proofs omitted in Section 5. In particular  we will give  a rigourous viscosity proof of the main existence result.\\

  {\em Acknowledgements.} We would like to thank  sincerely Luca Capogna for helpful discussions and for making a preprint  on the 
level set equation for horizontal mean curvature flow available  at an early stage, and we would like to express our thanks to  
Roberto Monti  for deep discussions and the proof of Lemma \ref{monti}. 
The second named author would like to thank Luca Mugnai for the constant support and the long useful conversations. 
This research was partially supported by the DFG through Forschergruppe 718.


\section{Mean curvature in sub-Riemannian geometries.}
\subsection{Sub-Riemannian geometries and Carnot groups.}$\quad$
\\
In this section we recall briefly what  sub-Riemannian 
geometries and Carnot groups are (for more information, see \cite{{bel},{CarnotGroup},{montgomery}}).\\
Let $X_1(x),...,X_m(x)$ a family of smooth vector fields  on $\R^n$ 
and
$$\cH_x=\mathrm{Span}(X_1(x),...,X_m(x))$$ then
$\cH=\{(x,v)\,|\,x\in \R^n,v\in \cH_x\}$ is a
\emph{distribution} on $\R^n$.
\begin{defi} A
\emph{sub-Riemannian metric} in $\R^n$ is a Riemannian metric
$\left<\cdot,\cdot\right>_g$ defined on the fibers of a distribution
$\cH$.
\end{defi}
An  absolutely continuous curve $\gamma:[0,T]\to \R^n$ is
called  \emph{horizontal}, if and only if,
$\dot{\gamma}(t)\in\cH_{\gamma(t)}$, a.e. $t\in [0,T]$, i.e.  
\begin{equation}
\label{safari}
\dot{\gamma}(t)=\sum_{i=1}^m \alpha_i(t)X_i(\gamma(t)),\quad \textrm{a.e.}\; t\in [0,T].
\end{equation}
 For any  
horizontal curve, we defined a length-functional  as
$$
l(\gamma)=\int_0^T |\dot{\gamma}(t)|_g dt,
$$
with
$|\dot{\gamma}(t)|_g=\left<\dot{\gamma}(t),\dot{\gamma}(t)\right>_g^{\frac
{1}{2}}$. From now to on, we  choose the Riemannian metric $\left<\cdot,\cdot\right>_g$ such that the vector fields $X_1,...,X_m$ are orthonormal, that means, by \eqref{safari}
$$
l(\gamma)=\int_0^T \sqrt{ \alpha_1^2(t)+...+\alpha_m^2(t)} \;dt,
$$
Once defined the length-functional we can introduce the following distance
\begin{equation}
\label{C-Cdistance}
 d(x,y):=\inf\{ l(\gamma)\,|\,
\gamma\;\textrm{horizontal curve joining $x$ to $y$}\}.
\end{equation}
Let $m(x)=\Dim( \Span(\cH_x)).$ If $m(x)=n$, at any point $x\in \R^n$, 
the distribution generates the whole tangent bundle  so the 
Riemannian metric $\left<\cdot,\cdot\right>_g$  induces by  \eqref{C-Cdistance} a Riemannian distance  on $\R^n$.
Otherwise, if $m<n$ at 
some point, this is not possible and  \eqref{C-Cdistance} can 
be infinite for some pairs $(x,y)$. The so-called
H\"ormander  condition, which we will explain below, 
guarantes that $d(x,y)$ remains finite.\\
Recall that the Lie bracket between
two vector fields $X$ and $Y$ is defined as the vector field which acts
on smooth real functions by
 $[X,Y]f=X(Yf)-Y(Xf)$. Let $\cL^1=\{X_1,...,X_m\}$, $\cL^2=\{[X_i,X_j]|
\,i,j=1,...,m\}$ and, for $k>2$,
$\cL^k=\{[Y_i,Y_j]\,|\,Y_i\in \cL^h,Y_j\in \cL^l,
h,l=1,...,k-1\}\backslash \bigcup_{i=1}^{k-1}\cL^i$, then the
\emph{Lie algebra} associated to the distribution $\cH$ is the set
$\cL=\bigcup_{i\in \N}\cL^i$.
\begin{defi}
We say that the distribution $\cH$ satisfies  the \emph{H\"{o}rmander
condition}  if there exists $r\ge 1$ so that $\cL=\bigcup_{i=1}^r\cL^i$ and
${\rm span}(\cL)=\R^n$ at any point.  
The number $r$ is called \emph{step} of the distribution.
In such a case, we call the function \eqref{C-Cdistance}, induced by the distribution $\cH,$ a \emph{sub-Riemannian} (or \emph{Carnot-Carath\'eodory}) 
distance on $\R^n$, and
the triple $(\R^n, \cH, \left<\cdot,\cdot\right>_g)$ 
is called a \emph{sub-Riemannian geometry}.
\end{defi}
The main consequence of the H\"{o}rmander
condition is that the associated sub-Riemannian
distance is finite (Chow's Theorem).
Moreover, the H\"{o}rmander condition implies  that the associated
distance is continuous in $\R^n$ with respect to the Euclidean topology, and
it implies the existence of minimizing geodesics. The geodesics in these geometries are usually not locally unique.\\

Carnot groups are particular sub-Riemannian geometries, where a 
structure of Lie  group is defined. 
We recall  briefly the main definitions.  For more details we refer to \cite{{capogna},{garofalo},{CarnotGroup}}.
Let $\G=(\R^n,\cdot)$ a Lie group and   $g$ the stratified Lie 
algebra of the left-invariant vector fields.
\begin{defi} A Carnot group is a Lie group $\G=(\R^n,\cdot)$, nilpotent and simply connected, whose Lie algebra $g$ admits a stratification, i.e. it can be written as $g=\oplus_{i=1}^rV_i$, where $V_i$ satisfy the property
$$
V_{i+1}=[V_1,V_i] \neq \emptyset  \ \textrm{for}\  i=1,\ldots,r-1 \quad \textrm{and}\quad [V_1,V_r]=\{0\} 
$$
 $\R^n$ endowed with the distribution $\cH=V_1$ and the Euclidean metric on it, is a sub-Riemannian geometry with step $r$.
\end{defi}
On a Carnot group we can define a family of dilations 
$\delta_{\lambda}(x)=\lambda^i x$, whenever $x\in V_i$. Note that
$
\norma{\delta_{\lambda} (x)}_{CC}=\lambda\norma{x}_{CC}
$,
where $\norma{x}_{CC}=d(x,0)$ is the Carnot-Carath\'eodory norm 
defined on $\G$.
Moreover, in any Carnot group, it is possible to define an homogenous 
norm, that is
$$
 \norma{x}_0:=\left( \sum_{i=1}^{r}
|x_i|^{\frac{2r!}{i}}\right)^{\frac{1}{2r!}}
$$
with $x=(x_1,...,x_k)$ and $x_i\in V_i$. This norm is indeed homogeneous w.r.t. the dilations defined on the 
Carnot group and it is equivalent to but more useful than the 
Carnot-Carath\'eodory norm. In 
fact, explicit calculations are easy and, if we define the homogenous 
distance, setting
$$
d_0(x,y)=\norma{y^{-1}\cdot x}_0
$$ 
where $y^{-1}$ is the inverse of $y$ w.r.t. the group multiplication $\cdot$,
then   
$d_0^{\alpha}$ is smooth for $\alpha=2r!,$ where $r$ is the 
step of the distribution.
By using the dilations and the group operation one can get 
far stronger results than in the general sub-Riemannian case.
In particular, it is possible to define intrinsic regularizations by inf- and 
sup-convolutions  (see \cite{Wang}, Definition 3.1 and Proposition 3.3) 
which imply many uniqueness results for visocity solutions of nonlinear equations which are  still open in the 
general case of H\"ormander vector fields. \\
We would like to conclude the section with some examples.
\begin{ex}[Heisenberg group]
\label{Ex:Heisenberg}
The most important sub-Riemannian geometry is the (1-dimensional) 
Heisenberg group $\dH^1$ which is the Carnot group  $(\R^3, \cdot)$, 
with law 
$$(x,y,t)\cdot(x',y',t')= \left(x+x',y+y', t+t' + \frac{xy'-yx'}{2} \right)
$$
and $(x,y,z)^{-1}=(-x,-y,-z)$.
Set $p=(x,y,z)\in \R^3$, the associated Lie algebra given by the left-invariant vector fields is $ X(p) =(1,0,-y/2)^T$ and $Y(p)=(0,1,x/2)^T$.
The  bracket relations are $[X,Y]=T=(0,0,1)^T$ and $[X,T]=[Y,T]=0$.
 Hence  $\dH^1$ is a sub-Riemannian geometry with step $r=2$.\\
The differential of the left-translations is $L_p:\dH^1\to \dH^1$, $q\to p
\cdot q$ in $\dH^1$
$$dL_p = \left( 
\begin{array}{ccc}
 1& 0 & 0\\
0& 1 &0  \\
-\frac y 2  & \frac x 2  & 1 
 \end{array}
\right)= - dL_p^{-1}
$$
The family of dilations is given by
$ \delta_{\lambda}(x,y,z)=(\lambda x,\lambda y,
\lambda^2
z)
$
and the homogeneous norm is
$
\norma{(x,y,z)}_0 =\big((x^2+y^2)^2
+z^2\big)^{\frac{1}{4}}
$.\\
In general the $n$-dimensional Heisenberg group $\dH^n$ is defined 
on $\R^{2n+1}$ and the group law, for any point $(x,y,z)\in \R^n\times
\R^n\times \R$, is given by $(x,y,z)\cdot(x',y',z')= (x+x',y+y', z+z' + \frac
{1}{2} (x\cdot y'-y\cdot x'))$ and all previous notions are still true.
\end{ex}
\begin{ex}[Gru\v{s}in plane]
\label{Ex:Grusin}
The Gru\v{s}in plane is the sub-Riemannian geometry
defined on $\R^2$,  by the distribution spanned by the two
vector fields
 $X_1(x,y)=(1,0)^T$  and $X_2(x,y)=(0,x)^T$. In this case the step of 
the distribution is 2 (as in the Heisenberg group) but the dimension of 
the distribution is not constant since it is  $m=1$ at the origin and 
$m=2$, otherwise.
Even if the Gru\v{s}in plane is not a Carnot group, its structure is not 
so different from 
the structure of $\dH^1$.
In fact, it is possible to define dilations and an homogeneous norm, 
  which are $\delta_{\lambda}(x,y)=(\lambda x,\lambda^2 y)$
and $\norma{(x,y)}_0=|x|+|y|^{\frac{1}{2}}$.
\end{ex}
\begin{ex}[Roto-translation geometry]
\label{Ex:Visual}
The roto-translations geometry is generated on $\R^3$ by  $X_1(x,y,
\theta)=(\cos \theta,\sin \theta,0)^T$ and $X_2(x,y,\theta)=(0,0,1)^T$. It 
is a 2-step sub-Riemannian geometry which was introduced by Citti 
and Sarti in  \cite{citti}, in order to study the modal and amodal 
perceptual completion of the visual cortex.
\end{ex}
%

\subsection{Horizontal mean curvature.}$\quad$\\
We introduce the   notion of horizontal mean curvature in sub-Riemannian manifolds. There are various ways how to define mean curvature in such spaces.  
For more details and the links with the variation of the area  and the approximation  by corresponding Riemannian objects, we refer to \cite{capogna} in the Heisenberg group and \cite{{capognaCitti},{garofalo},{pauls}} in a more general context.\\

Given $X_1,....,X_m$   smooth vector fields on $\R^n$, satisfying the H\"ormander condition, we indicate with $M$  the associated sub-Riemannian manifold $(\R^n,\cH,\left<\cdot,\cdot\right>_g)$ and we recall that $\left<\cdot,\cdot\right>_g$ is built in such a way that $X_1,...,X_m$ are orthonormal, i.e. for  $v,w$ horizontal vectors, 
$$
\left<v,w\right>_g=\left<\alpha,\beta\right>
$$
where $\alpha$ and $\beta$ are the coordinate-vectors of $v$ and $w$ w.r.t. $X_1,...,X_m$,  and $\left<\cdot,\cdot\right>$ is the usual inner product in $\R^m$.\\
First we recall that the \emph{horizontal gradient}  of a function $u: M\to \R$ is the horizontal vector field defined as
$$
\Upsilon u (x) =(X_1 u) X_1(x)+...+(X_m u)X_m(x)\in \R^n
$$
From now  on, we will  often omit the  dependency on the point $x$ and use the  {\emph coordinate-vector field}   of $\Upsilon u $ w.r.t. $X_1,...,X_m$, that is
$$
\cX u  =(X_1 u,..,X_m u)^T\in \R^m
$$
Note that 
$$|\Upsilon u |^2_g= \sum_{i=1}^m   \big(X_i u\big)^2=|\cX u|^2$$
where $|\cdot|$ is the eculidean norm in $\R^m$.\\

Before giving the main definitions we want also to point out the following notation. We are interested in the study of hypersurface on sub-Riemannian manifolds so we have to treat two different kinds of tangent spaces: the tangent space of the manifold and the tangent space to the hypersurface. \\
 Fix a point $x\in M.$ In order to avoid confusion, we call \emph{horizontal space} 
the tangent space of the sub-Riemannian manifold, denoted by $H_x M,$ while the \emph{tangent space} and \emph{horizontal tangent space} are, respectively, the Euclidean tangent space of the hypersurface  $\Sigma \subset M$ and the intersection of the Euclidean tangent space with the horizontal space.
We indicate the latter two objects by $T_x \Sigma$ and $HT_x \Sigma$.
\begin{defi}
 Let  $\Sigma=\{u=0\}$ a hypersurface in $M$, we call   \emph{horizontal normal of $\Sigma$} the renormalized projection of the Euclidean normal on the horizontal space, which is
$$
n_0(x)=\frac{\Upsilon u}{|\Upsilon u|_g}
$$
We introduce the \emph{horizontal mean curvature}  as the horizontal divergence of the horizontal normal:
\begin{equation}
\label{limite1}
 k_0(x) :=\sum_{i=1}^m X_i\left( \frac{ X_i u}{|\cX u|}\right)
\end{equation}
\end{defi}
Unlike in the Euclidean case, the horizontal normal to a smooth hypersurface is not always well defined. In fact, whenever the Euclidean normal is ``vertical'', which means that its projection on the horizontal space vanishes,  then  $n_0$ and hence $k_0$ are  not defined. 
\begin{defi}
Given a hypersurface $\Sigma=\{u=0\}\subset M$, we call  \emph{set of the characteristic points} the set of the points where the Euclidean normal is perpendicular to the horizontal  space, that is
\begin{equation}
\label{characteristSet}
 char(\Sigma)\!=\!\{x\in M\,|\, H_x M\subset T_x\Sigma\}\!=\!\{x\in M\,|\, HT_x\Sigma=H_x M\}\!=\!\{x\in M\,|\,|\cX u|=0\}
\end{equation}
\end{defi}
As we will see later, the existence of characteristic points make the evolution by horizontal mean curvature flow much different from the corresponding Euclidean or Riemannian evolution.\\

 Let $u:\R^n\to \R$ be a smooth function.
In the Euclidean setting it is easy to show that the mean curvature of $\Sigma=\{x\in \R^n:\ u(x)=0\}$ is equal to the Laplacian minus the infinite-Laplacian, 
both divided by the modulus of the gradient. We next recall the definition of the corresponding horizontal operators and we  show that this fact is still true at the non characteristic points. 
First we recall that the symmetrized matrix of second derivatives is a $m\times m$ matrix defined as 
$$(\cX^2u)^*_{i,j}=\frac{X_i(X_ju)+X_j(X_iu)}{2}$$
We call  \emph{horizontal Laplacian} and \emph{horizontal  infinite-Laplacian},   respectively, the following second order operators:
 $$
\Delta_{0} u= \sum_{i=1}^m  X_i (X_i u),\quad \Delta_{0,\infty}u=\bigg<
(\cX^2u)^*\frac{\cX u}{|\cX u|},\frac{\cX u}{|\cX u|}\bigg>
 $$
 Then, as in the Euclidean and Riemannian case, it is immediate to show that
\begin{equation}
\label{HorMC}k_0(x) =|\cX u|^
{-1
}\big(\Delta_{0} u-\Delta_{0,\infty}u\big)
\end{equation}
For later use, we express all the previous objects by  the matrix associated to the 
sub-Riemannian geometry, the Euclidean gradient $Du$ and the Euclidean Hessian $ D^2u.$ So let 
 $\sigma(x)$   be  the smooth  $m\times n$  matrix defined as $\sigma(x)=[X_1(x),....X_m(x)]^T$,   then  
 the coordinate-vector of the horizontal gradient can be expressed as 
$$
\cX u(x)=\sigma(x) D u(x).
$$
The main point is to  express the symmetrized matrix of horizontal second derivatives, uing $\sigma(x)$.
In fact, the matrix $\big(\cX^2u\big)^*$ does not depend on just second order derivatives like the corresponding Euclidean one but also on first order derivatives coming from the  derivatives of the vector fields. To be more precise, one can write 
\begin{equation}
\label{panino}
\big(\cX^2u\big)^*=\sigma(x) (D^2u)\sigma^T(x)+A(X_1,...,X_m, Du)
\end{equation}
where the matrix $A$ is a symmetric $m\times m $ matrix defined as
\begin{equation}
\label{matricial}
A_{i,j}(X_1,...,X_m, Du)=\frac{1}{2} \left<\nabla_{X_i}X_j+\nabla_{X_j}X_i, Du\right>,
\quad \textrm{for}\; i,j=1,...,m
\end{equation}
and $\nabla_{X_i}X_j$ is the (Euclidean) derivative of the vector field $X_j$ w.r.t. the vector field $X_i$. Hence, it is possible to  rewrite   previous  horizontal second order operators as
\begin{equation}
\label{Laplacian-matrix}
\Delta_0 u= \tr\big(\sigma(x)( D^2u)\sigma^T(x)\big)+\sum_{i=1}^m \left<\nabla_{X_i}X_i, D u\right>
\end{equation}
and 
\begin{align}
\label{InftyLaplacian-matrix}
\Delta_{0,\infty} u
&= \bigg<\big(\sigma(x) (D^2u)\sigma^T(x)\big)\frac{\sigma(x) Du}{|\sigma(x) Du|},\frac{\sigma(x) Du}{|\sigma(x) Du|}
\bigg>\\
&+ \bigg<A(X_1,...,X_m, Du)\frac{\sigma(x) Du}{|\sigma(x) Du|},\frac{\sigma(x) Du}{|\sigma(x) Du|}
\bigg>
\end{align}
This paves the way for studying the horizontal mean curvature flow
by the techniques from stochastic control theory which we explain later.\\

We conclude  this section looking at the particular case of  the Heisenberg group. 
It is known  that in general $\nabla_{X_i}X_j$ is perpendicular to the horizontal space (see \cite{garofalo} for a proof in Carnot groups). An easy calculation  shows  that in the Heisenberg group $\nabla_{X_i}X_j+   \nabla_{X_j}X_i=0$, for every $i,j=1,2$.  So 
$$
\Delta_0 u= \tr\big(\sigma(x)( D^2u)\sigma^T(x)\big) \;
\textrm{and}\; 
\Delta_{0,\infty} u= \bigg<\big(\sigma(x) (D^2u)\sigma^T(x)\big)\frac{\sigma(x) Du}{|\sigma(x) Du|},\frac{\sigma(x) Du}{|\sigma(x) Du|}
\bigg>.
$$
That makes  it easier to study  explicit examples in the Heisenberg group.
 Hence, let $\Sigma=\big\{(x,y,z)\in \dH^1|u(x,y,z)=0\big\}$  be a surface in $\dH^1$, then we can explicitly calculate 
all previous quantities and, in particular, the set of characteristic points  becomes:
\begin{equation}
\label{characteristicPoints1}
char(\Sigma)=\bigg\{(x,y,z)\in \Sigma  \;\big|\;\left(u_x  -\frac{y}{2}u_z\right)^2+\left(u_y+ \frac{x}{2}u_z\right)^2=0\bigg\}\end{equation}
Whenever   $\Sigma$ is invariant by rotation around the  
$z$-axis, i.e. $u(x)= |z|-f(r)$,  with $r=\sqrt{x^2+y^2}$, we get
\begin{equation}
\label{radialMCH}
k_0(x,y,z)= \pm\frac{\frac{1}{4}r^2 f''(r)+\frac{(f'(r))^3}{r}}{((f'(r))^2+\frac
{1}{4}r^2)^{\frac{3}{2}}}
\end{equation}
(depending if $z>0$ or $z<0$). In such a case  the set of the 
characteristic points is
\begin{equation}
\label{characteristicPoints2}
char(\Sigma)=\left\{(x,y,z)\in \Sigma \;\big|\;4\big(f'(r)\big)^2 +r^2  =0
\right\}
\end{equation}
Obviously the   only possible solutions are $r=0$ and $f'(0)=0$.\\ 
That means that the unique possible characteristic points are the flat 
intersection with the $z$-axis, i.e.  $(0,0, \pm f(0))$, with Neuman 
boundary condition $f'(0)=0$. Using \eqref{radialMCH}, it is easy to calculate the horizontal mean curvature in the following examples.
\begin{ex} [\cite{capogna}]
\label{variuosEx}$\quad$\\
\begin{description}
\item[(1)] {\bf Euclidean ball.} If $\Sigma=\{(x,y,z)\in \R^3| x^2+y^2
+z^2=R^2\}$,
$$
k_0=\frac{2(4+R^2)}{\sqrt{x^2+y^2}(4+z^2)^{\frac{3}{2}}}.
$$
The characteristic  points are $(0,0,\pm R)$.
\item[(2)] {\bf Kor\'anyi ball.} If $\Sigma=\{(x,y,z)\in \R^3| (x^2+y^2)^2
+16 z^2=R^4\}$, 
$$
k_0=\frac{3\sqrt{x^2+y^2}}{R^{2}}.
$$
The characteristic  points are $(0,0,\pm \frac{R^2}{4})$.
\item[(3)] {\bf Heisenberg ball.} Let $\Sigma=\{(x,y,z)\in \R^3| d((x,y,z),
(0,0,0))=R^2\}$, then, using the  explicit formula for the Heisenberg 
geodesics, we have\\ $r=\sqrt{x^2+y^2}=\frac{2}{c} \sin(cR/2)$ and $z=
\frac{cR-\sin (cR)}{2c^2}$, therefore
$$
k_0=\frac{1}{2}\;\frac{c/2}{\sin(cR/2)}\;\frac{\sin(cR)-cR\cos(cR)}{\sin
(cR/2)-(cR/2)\cos(cR/2)}.
$$
The characteristic points are 
$(0,0,\pm \frac{R^2}{4 \pi})$.
\end{description}
\end{ex}
The situation is particularly easy when there are no characteristic points. 
\begin{defi}
We call \emph{regular hypersurface} any $C^1$ hypersurface such that all the points are 
not characteristic.
\end{defi}
In Riemannian geometries any $C^1$ hypersurface is regular,
while in the sub-Riemanian case very few $C^1$ hypersurfaces are.
We quote the  following remark due  by Roberto Monti.
\begin{lemma}
\label{monti}
Any $C^1$ compact surface $\Sigma\subset \dH^1$, topologically 
equivalent to the sphere, is not regular.
\end{lemma}
\begin{proof}
The ``hairy ball theorem'' from algebraic topology states 
that, given a vector field tangent 
to a  surface in $\R^3$, topologically equivalent to the sphere, there 
exists at least  one point where the vector field vanishes.
Let us consider as vector field the horizontal normal vector. 
Assuming that the surface is regular, such a vector 
field is different from zero at any point. Now we can (e.g. using the complex 
interpretation of the Heisenberg group) just rotate such a 
vector by $\frac{\pi}{2}$. This new vector field is still not vanishing at 
any point but it is tangent to the surface, which contradicts the 
topological theorem.
\end{proof}
However it is possible to find some examples of (non-compact) regular surfaces 
among the rotational surfaces $\{|z|=f(\sqrt{x^2+y^2}) 
\,|,x^2+y^2\in [a,b]\}$ with $a,b\in \R$. Whenever $a>0$ (i.e. $r\neq0$ 
in whole the surface), there are no characteristic points.
So rotational surfaces around the $z$-axis are regular, provided they do not intersect the $z$-axis.
That remark leads to the following examples.
\begin{ex}
\label{regularSurface}
Regular surfaces are:
\begin{enumerate}
\item any vertical plane $ax+by=d$,
\item any cylinder around the $z$-axis,
\item any torus around the $z$-axis.
\end{enumerate}
\end{ex}
Let us point out that non-regular surfaces are the far more interesting examples, because all sphere-type surfaces are not regular (Lemma \ref{monti}) and, moreover, the characteristic points are what really makes this geometry so geometrically different from the analogous Euclidean one.

\section{Generalized evolution by horizontal mean curvature.}
\label{generalizedDefinitions}

In Euclidean spaces, the motion by mean curvature flow of a manifold of codimension 1 is the geometrical evolution defined by requiring the normal velocity
at each point of the manifold.
 Only few results are known for mean curvature flow in  sub-Riemannian manifolds, 
 i.e. for the evolution obtained by replacing all the geometrical  objects by
the corresponding horizontal quantities.  
In these degenerate spaces, such a kind of evolution is very different from the corresponding Euclidean motion, especially because of the existence of characteristic points, i.e. points were motion in the (Euclidean) normal direction
is not ``admissible''.\\

Let us define  rigorously the evolution (or motion)  by mean curvature flow (MCF)
in a sub-Riemannian geometry. We give first a notion  assuming that the hypersurface is {\em regular}  (i.e. smooth without characteristic points)
and then we  derive a weak notion holding for every hypersurface.
\begin{defi}
\label{evDefi1hor}
For $t>0$, let
 $\Gamma (t)$ a family of regular hypersurfaces in a sub-Riemannian geometry $(\R^n, \cH, \big<\cdot,\cdot\big>_g)$.  
We say that  $\Gamma (t)$ is an  evolution by  horizontal mean curvature flow of the hypersurface $\Gamma_0$ if and only if the following holds: $\Gamma(0)=\Gamma_0,$ and  for any smooth horizontal
curve $x(t): [0,T]\to \R^n$  such that
$x(t)\in \Gamma(t)$ for all $t\in [0,T]$, the ``horizontal normal velocity'' is equal to minus the horizontal mean curvature, i.e. 
\begin{equation}
\label{regularMCEhor}
v_0(x(t)):=\left<\dot{x}(t), n_0(x(t))\right>_g=-k_0(x(t))
\end{equation}
where $n_0(x(t))$ and $k_0(x(t))$ are the horizontal external normal and the horizontal curvature of $\Gamma(t)$, calculated at the point $x(t)\in \Gamma(t)$.
 \end{defi}
Note that $\left<\dot{x}(t), n_0(x(t))\right>_g$ is well defined since $x(t)$ is horizontal and smooth and we assume that $\Gamma(t)$ is a regular 
hypersurface.\\
  \eqref{regularMCEhor} is not sufficient to describe the evolution since, like in the Euclidean case,
it is not defined whenever the hypersurface develops singularities (which can happen in the Euclidean case starting from a smooth hypersurface) 
and it is not defined at the characteristic points, which are a specific feature of the sub-Riemannian MCF.
\\
We introduce a weak notion of evolution by mean curvature flow, using the level set approach. Such a definition was given first by   Chen, Giga and Goto  \cite{Giga}  and, independently, by Evans and Spruck \cite{EvSp:91}. It is based on the idea of defining the evolution of a function $u(t,x)$ by a degenerate parabolic PDE in such a way that each level set $\{x\in \R^n:\ u(t,x)=c\}$ evolves by mean curvature as long as it is a smooth manifold, see e.g. \cite{EvSp:91}. Exploiting the fact that this PDE is degenerate parabolic, one can define  a generalized solution, called viscosity solution. \\

Next we derive  this degenerate PDE for regular hypersurfaces.\\
Let $\Gamma(t)=\{u(t,x)=c\}$, then the horizontal normal of $\Gamma(t)$ at $x(t)$ is  given  by 
$n_0(x(t))=\Upsilon u/|\Upsilon u|$. Since $x(t)$ is horizontal and smooth,
$$
\dot{x}(t)=\sum_{i=1}^m \alpha_i(t)X_i(x(t))=\sigma^T(x(t)) \boldsymbol{\alpha}(t), \quad \textrm{for every}\;\; t\in [0,T].
$$
Set $\boldsymbol{\alpha}(t)=(\alpha_1(t),...,\alpha_m(t))^T$ and  
recalling that  $\left<\dot{x}(t), \frac{\Upsilon u}{|\Upsilon u|}\right>_g=\left<\boldsymbol{\alpha}(t), \frac{\cX u}{|\cX u|}\right>_m$,  the horizontal normal velocity can be written as
\begin{equation*}
\left<\dot{x}(t), n_0(x(t))\right>_g
=\left<\boldsymbol{\alpha}(t), \frac{\cX u}{|\cX u|}\right>_m
=\!\left<\boldsymbol{\alpha}(t), \frac{\sigma(x) Du}{|\sigma(x) Du|}\right>_m
=|\cX u|^{-1}\!\left<\dot{x}(t), D u\right>_n
\end{equation*}
with $\left<\cdot,\cdot\right>_m$ and $\left<\cdot,\cdot\right>_n$   denoting the inner product in $\R^m$ and $\R^n$  respectively.\\
From now to on, we can procede similarly to the Euclidean case. In fact, $x(t)\in \Gamma(t)$ if and only if $u(t,x(t))=c$. Taking
the derivative in time  and using \eqref{regularMCEhor}, yields
\begin{equation}
\label{pizza1}
u_t(t,x(t))=-\left<\dot{x}(t), D u(t,x(t))\right>_n=-|\cX u|\left<\dot{x}(t), n_0(x(t)\right>_g=|\cX u|k_0(x(t))
\end{equation}
  It remains to use $k_0(x(t))=\sum_{i=1}^m X_i\left( \frac{ (\cX  u)_i}{|\cX u|}\right)
$, which gives 
\begin{equation}
\label{levelSetMCE}
u_t=\tr \left((\cX^2 u)^*\right)-\bigg<(\cX^2 u)^*\frac{\cX u}{|\cX u|},\frac{\cX u}{|\cX u|}\bigg>=\Delta_0 u-\Delta_{0,\infty} u.
\end{equation}
We want to point out that  equation  \eqref{levelSetMCE} is parabolic  degenerate whenever  $\cX u=\sigma(x) D u=0$. 
  We call  the points where the horizontal gradient vanishes \emph{singularities}.
In the Euclidean case it is known that 
singularities can lead to the so-called fattening of level sets. 
We say fattening  occurs when the level set has no-empty interior, that means in particular that the gradient vanishes in an open subset, i.e. the co-dimension of
the level set is locally zero
(see \cite{{AAG95},{BePa:95},{Gigabook}}, for more information).  
In the sub-Riemannian geometry, singularities are related to the vanishing of the horizontal gradient, which happens at characteristic points. \\
Note that the co-dimension of the horizontal tangent space is not zero at a characteristic point.
Therefore, singularities of the level set equation in sub-Riemannian geometries describe very different geometrical phenomena in spite of the formal analogy with the Euclidean case.\\ 
In order to introduce a generalized  motion by horizontal mean curvature,
we follow the  definition introduced by Chen, Giga and Goto in  \cite{Giga} for the Euclidean evolution and by Giga in \cite{Gigabook} for generic degenerate parabolic equations. \\

Recall that the structure of \eqref{levelSetMCE} is that of a degenerate parabolic equation 
\begin{equation}
\label{mariapia}
u_t+F(x,Du,D^2 u)=0
\end{equation}
with 
$$
F(x,p,S)= -\tr\big( \sigma(x) S\sigma^T (x)+A(x,p) \big)+\left< \big( \sigma(x) S\sigma^T (x)+A(x,p) \big)\frac{\sigma(x)p}{|\sigma(x)p|},\frac{\sigma(x)p}{|\sigma(x)p|}\right>
$$
with $A(x,p)$ defined in  \eqref{matricial}. For sake of semplicity, set
$$
\widetilde{S}=\sigma(x) S \sigma^T (x)+A(x,p),
$$
  then we can easily calculate that the upper and  lower  semicontinuous envelopes of  equation \eqref{mariapia} are
$$
\left\{
\begin{aligned}
-\tr\big(\widetilde{S}\big)+\left< \widetilde{S}\frac{\sigma(x)p}{|\sigma(x)p|}, \frac{\sigma(x)p}{|\sigma(x)p|}\right>, \quad &|\sigma(x)p|\neq 0
\\
-tr\big(\widetilde{S}\big)+\lambda_{max}\big( \widetilde{S} \big), \quad &|\sigma(x)p|=0
\end{aligned}
\right.
$$
and
$$
\left\{
\begin{aligned}
-\tr\big(\widetilde{S}\big)+\left< \widetilde{S}\frac{\sigma(x)p}{|\sigma(x)p|}, \frac{\sigma(x)p}{|\sigma(x)p|}\right>, \quad &|\sigma(x)p|\neq 0
\\
-tr\big(\widetilde{S}\big)+\lambda_{min}\big( \widetilde{S} \big), \quad &|\sigma(x)p|=0
\end{aligned}
\right.
$$
where $\lambda_{\max}(S)$ and $\lambda_{\min}(S)$ are the  maximal and minimal eigenvalues of $S$.\\

Taking $\widetilde{S}=(\cX^2 u)^*$ and $\sigma(x)p= \cX u$, we 
can give the following definition for the generalized motion by horizontal mean curvature flow.
\begin{defi}
\label{giga}
Let $\Gamma_0=\{x\in \R^n|u_0(x)=0\}$ hypersurface in $\R^n$. We say that $\Gamma(t)=\{x\in \R^n|u(t,x)=0\}$  is \emph{generalized evolution by horizontal mean curvature flow} if $u$ satisfies the initial condition $u(0,x)=u_0(x)$ and it is a  viscosity solution of  \eqref{levelSetMCE}   in the sense of \cite{Gigabook}, that means $u$ is a continuous function and
\begin{enumerate}
\item  for any 
 $\varphi\in \C^2(\R^n\times (0,+\infty))$ such that $u-\varphi$ has a local minimum  at $(t_0,x_0)$, then
 \begin{equation}
 \label{levelSetMCEsuper}
 \left\{
\begin{aligned}
\varphi_t-\Delta_0 \varphi+\Delta_{0,\infty}  \varphi \geq 0, \quad &\textrm{at} \,(t_0,x_0),\, \textrm{if}\; \;\cX \varphi(t_0,x_0)\neq 0\\
\varphi_t-\Delta_0 \varphi+\lambda_{\max}( (\cX^2 \varphi)^*)\geq 0,\quad &\textrm{at} \,(t_0,x_0),\, \textrm{if}\; \;\cX \varphi(t_0,x_0)=0
\end{aligned}
\right.
\end{equation}
\item for any 
 $\varphi\in \C^2(\R^n\times (0,+\infty))$ such that $u-\varphi$ has a local maximum  at $(t_0,x_0)$, then
 \begin{equation}
 \label{levelSetMCEsub}
 \left\{
\begin{aligned}
\varphi_t-\Delta_0 \varphi+\Delta_{0,\infty}  \varphi \leq 0, \quad &\textrm{at} \,(t_0,x_0),\, \textrm{if}\; \;\cX \varphi(t_0,x_0)\neq 0\\
\varphi_t-\Delta_0 \varphi+\lambda_{\min}( (\cX^2 \varphi)^*) \leq 0,\quad &\textrm{at} \,(t_0,x_0),\, \textrm{if}\; \;\cX \varphi(t_0,x_0)=0
\end{aligned}
\right.
\end{equation}
\end{enumerate}
\end{defi}
As we will see later, we can give the same definition for discontinuous functions, requiring the subsolution condition  
(Def. \ref{giga}, (2)) 
for the upper semicontinuous envelope of $u$ and, the supersolution condition  
(Def. \ref{giga}, (1)) 
for the lower semicontinuous envelope of $u$.\\

We would like to point out that  the level set approach gives a well-posed notion of evolution, provided that  the set $\Gamma(t)$ does not depend on the chosen parametrization $u_0$ but just on the level set $\Gamma_0$. This
is the case if  whenever $$
U^+_0:=\{x\in \R^n\,|\,u_0(x)\le 0\}\subset \{x\in \R^n\,|\,v_0(x)\le 0\}=:V^+_0,$$ indicating by   $u(t,x)$  and $v(t,x)$, respectively, the viscosity solutions of equation  \eqref{levelSetMCE} with initial conditions $u_0$ and $v_0$,  then  $$\{x\in \R^n\,|\,u(t,x)\le 0\}\subset \{x\in \R^n\,|\,v(t,x)\le 0\}$$ (similarly for the nonnegative level sets).
In the Euclidean case, it is possible to prove this,  using comparison principles  for the level set equation (and  a suitable reparametrization for the initial data   such that we obtain  $u_0\ge v_0$
see \cite{Gigabook}, Theorem 4.2.8. for a complete proof).\\
The proof holds for equations with $F$ strong geometric, provided a comparison principle for viscosity solutions holds. 
In our case $F$  is strong geometric, 
hence the well-posedness of Definition \ref{giga} depends  mainly on the existence of comparison principles. Unfortunately, very little is known about  comparison principles (and hence uniqueness in the case of the evolution by horizontal mean curvature. In \cite{capognaCitti}, Capogna and Citti prove comparison principles 
in Carnot groups for special classes of initial data:
\begin{teo}[\cite{capognaCitti}, Theorems 3.1]
\label{unicita}
Let $G$ be a Carnot group and $u$ and $v$ bounded viscosity subsolution and supersolution of \eqref{levelSetMCE}, respectively, with initial datum $u_0$ and $v_0$. Suppose 
$u_0$ or $v_0$ are uniformly continuous and for any $(x_H,x_V),(x_H,y_V)\in \G=\oplus_{i=1}^k V_i$, where we indicate by $z_H\in V_1$ the horizontal part of a point and by $z_V$  the vertical part (i.e. $z_V\in \oplus _{i=2}^k V_i$), then $u_0(x_H,x_V)\leq v_0(x_H,y_H)$.
Then $u(t,x)\leq v(t,x)$, for any $t\geq 0$ and $x\in \G$.
\end{teo}
The assumption  on the initial datum is used to avoid the problems created by the characteristic points.
They can  circumvent the degeneracy of the equation in the spirit of the Euclidean proof, using the method  of doubling variables and showing that the horizontal gradient of the chosen test function does not vanish at the maximum point. 
Spheres, tori and any compact surfaces (see Example \ref{variuosEx} for some of those) are not covered by the result of \cite{capognaCitti}.
\\
Let us point out that the definition introduced in \cite{capognaCitti} looks slightly different from ours.  
Indeed Capogna and Citti follow the  definition of viscosity solutions for degenerate parabolic equations used by Evans and Spruck in \cite{EvSp:91},  that is  a more general notion of solution. So comparisons for the Capogna-Citti's definition imply comprisons for our solutions.
The two definitions are equivalent in the Euclidean case (see  \cite{Gigabook}),   while this equivalence is not clear in the sub-Riemannian case. We would like to remark that the results proved for general nonlinear degenerate parabolic equations in  \cite{Gigabook} (like  equivalence of the definitions, comparison principles, existence, etc.) rely on techniques which are not applicable  in our case.
The main difference between the usual degenerate parabolic equations and the level set equation for the evolution by horizontal mean curvature flow 
is that equation \eqref{levelSetMCE} is discontinuous at the points 
$(p,x)\in \R^n\times \R^n$ such that $\sigma(x)p=0,$ which is a space-variable-depending set which has  {\em non-zero dimension} in $p$. \\

To conclude this section we are going to have a closer look at the case of the Heisenberg group.
Since $(\cX^2 u)^*=\sigma(x)(D^2 u)\sigma^T(x)$ is a symmetric $2\times 2$ matrix, there are exactly two eigenvalues. 
As the trace is the sum of the eigenvalues, we can  rewrite  Definition \ref{giga} at singular points as
$$
\varphi_t-\Delta_0 \varphi+\lambda_{\max}( (\cX^2 \varphi)^*)=\varphi_t-\lambda_{\min}( (\cX^2 \varphi)^*)\geq 0
$$
and
$$
\varphi_t-\Delta_0 \varphi+\lambda_{\min}( (\cX^2 \varphi)^*)=\varphi_t-\lambda_{\max}( (\cX^2 \varphi)^*)\leq 0
$$
   
In the particular case of  rotational surfaces around the $z$-axis in the Heisenberg group, it is easy to see that the level set equation  is continuous up to the set of characteristic points.
Let $\Gamma_0=\{(x,y,z)\in \dH^1\,|\, |z|
=f(r)\}$, i.e. $u(t,(x,y,z))=|z|-f(t,r)$ with $r=\sqrt{x^2+y^2}$. 
The level set equation at the non characteristic points with $0<r\ll1,$ is equal to
\begin{equation}
\label{levelsetRot}
f_t=
\frac{4 (f'(r))^3+r^3f''(r)}{4r(f'(r))^2+ r^3}
\end{equation}
We already remarked that the characteristic points correspond to $r=0$ with vanishing Neuman condition $f'(0)=0$. 
So  by the first-order Taylor expansion for the function $f'(r)$ in $0$ (that is
$
f'(r)=f'(0)+f''(0)r+o(r)=f''(0)r
$),
we can deduce
$$
\lim_{r\to 0} \frac{f'(r)}{r}=f''(0)
$$
Hence, at the characteristic points,  the matrix $(\cX^2 u)^*(r=0)$ has an eigenvalue equal to $f''(0)$ with multiplicity 2. So in such a case
the level set equation is continuous and we obtain  $f_t -f''(0)=0$, whenever $r=0$.   In particular, there is a rich class
of examples for which the velocity in the characteristic point is non-zero.
\section{Controlled diffusion processes.}

Let us first recall some elementary facts from stochastic analysis for continuous semi-martingales which can be found in any standard textbook such as e.g. \cite{MR1121940}. Given a probability space $(\Omega,{\mathcal F}, \dP) $ together with a  filtration  
$ \{{\mathcal F}t\}_{t\geq 0}$ let
$\xi(t)$ be continuous and adapted (i.e. $\xi(t)$ is ${\mathcal F}_t$-measurable), and let $B(t)$ be a Brownian motion adapted to the filtration.
Then 
{
the It\^o integral  $\xi dB(t)$is defined as the following  limit (as the step size of the partition decreases) in  $L^2(\Omega):$
$$
 \int_0^t \xi(s)dB(s)\;\;\large{:=\!\!\!\!\!^{^{L^2}}} \lim_{N \to +\infty} \sum_{i=1}^N \xi(t_i)\big(B(t_{i+1})- B(t_{i})\big).
 $$Note that this holds actually in a far more general setting: The convergence holds 
 in the space of continuous square integrable martingales, the deterministic
 partition may be replaced by one constructed via stopping times, the  integrand $\xi$ need not be continuous, but merely previsible,
 and the Brownian motion as integrator can be replaced by any square-integrable continuous (semi-) martingale $\eta(t).$ In latter, more
 general case, we write $(\xi d\eta)(t)$ for the It\^o-integral. 
 
The Stratonovich integral $\xi\circ d\eta$ is defined as
 $$
 \int_0^t \xi(s)\circ d\eta(s) \;\;\large{:=\!\!\!\!\!^{^{L^2}}} \lim_{N \to +\infty} \sum_{i=1}^N \frac{\xi(t_i)+\xi(t_{i+1})}{2}\big(\eta(t_{i+1})- \eta(t_{i})\big),
 $$
both integrals are related by the formula
$$
\xi\circ d\eta = 
\xi d\eta + \frac 1 2 d\langle \xi,\eta\rangle 
$$
where
$\langle \xi,\eta\rangle $ denotes the quadratic covariation of the processes $\xi$ and $\eta$ which is defined as
 $$
  \int_0^t d \langle \xi, \eta\rangle (s) \;\;\large{=\!\!\!\!\!^{^{L^2}}}  \lim_{N\to +\infty} \sum_{i=1}^N \big(\xi(t_{i+1})- \xi(t_{i})\big) \big(\eta(t_{i+1})- \eta(t_{i})\big)
 $$ 
The chain rule  looks classical if we use the Stratonovich integral. In fact, for any smooth $f$, the process $f(\xi(t))$ 
satisfies
$$
 d\big[f(\xi(t))\big]  = f'(\xi(t)) \circ d\xi,
 $$ which can be re-written as
 \begin{align*} 
 d\big[f(\xi(t))\big] 
& = f'(\xi(t)) d\xi + \frac 1 2 f''(\xi(t))d \langle \xi,\xi\rangle.
\end{align*} 
Note that, whenever $\xi=B$ is a Brownian motion, we have $d\langle \xi,\xi\rangle=d\langle B,B\rangle=dt$ and the formula above is the well known It\^o formula. This establishes the basic connection between second order PDE and stochastic processes which yields an extension of the classical method of characteristics to the case of second order equations.\\ 
We would like to  point out that
we will use the Stratonovich calculus for defining our controlled stochastic processes since, because the chain rule is the classical one, it does not 
depend on the chosen parametrization and so it is intrinsic in   Riemannian and sub-Riemannian geometries (see e.g. \cite{Hsu}). Nevertheless the It\^o calculus will be very useful for proofs and computations (see Sec. 5).\\

\subsection{The stochastic control problem.}$\quad$\\

It is well known that viscosity solutions of certain second-order equations are closely related to the value function of
stochastic control problems, see e.g. \cite{sonerbook}.  The relation between solutions of degenerate equations like in Definition
\ref{giga} and stochastic control problems is more complicated. Nevertheless,
Soner and Touzi (in \cite{{soner3},{soner1}}) and, using another approach, Buckdahn,  Cardaliaguet and Quincampoix 
(in \cite{buckdahn}) derived a stochastic representation for a set evolving by mean curvature flow (in the Euclidean case).
The following stochastic optimal control problem (\ref{reachSet}) has very much in common with its Euclidean predecessors \cite{soner3,soner1,buckdahn}, where now we have to replace the It\^o by a Stratonovich SDE which reflects the fact that  we do not work  in an Euclidean space.
\\ 

  Let  $(\Omega, \mathcal F, \{\mathcal F_t\}_{t\geq 0}, \dP)$ be a filtered probability space and 
$B$ is a $m$-dimensional Browinian motion adapted to the filtration $\{\mathcal F_t \}_{t\geq 0}$.
We define the set of admissible controls by 
\begin{equation*}
\label{ControlSet}
\cV=\{ (\cv(s))_{s\geq 0} \mbox{ predictable } | \,\cv(s) \in S_m,\cv\geq 0, I_m-\cv^2\geq 0, \tr(I_m-\cv^2)=1\}.
\end{equation*}
 Under suitable assumptions, each $\cv(s)$  determines a  (unique)control path  $\xi^{t,x,\cv(\cdot)}$ as a solution to the SDE
\begin{equation}
\label{SDEs-Srat}
\left\{
\begin{aligned}
&d\xi^{t,x,\cv(\cdot)}(s)=\sqrt{2}\sigma^T(\xi^{t,x,\cv(\cdot)}(s))\circ  d B^{\cv}(s), \quad s\in (t,T]\\
&d  B^{\cv}(s) =  {\cv}(s)dB(s), \quad\; \hspace{3cm} s\in (t,T]\\
&\xi^{t,x,\cv(\cdot)}(t)=x
\end{aligned}
\right.
\end{equation}
 where $\circ dB^{\cv}$ denotes the integral w.r.t to $B^{\cv}$ in the sense of Stratonovich.
  Using the relation $\xi\circ d\eta = \xi d\eta + \frac 1 2 \langle \xi, \eta \rangle$ between  the Stratonovich and the It\^o formulation, we get the following equivalent It\^o formulation for SDE \eqref{SDEs-Srat} 
\begin{equation}
\label{SDEs-Ito}
\left\{
\begin{aligned}
d\xi^{t,x,\cv(\cdot)}(s)&=\sqrt{2}\sigma^T(\xi^{t,x,\cv(\cdot)}(s)) \cv(s) d B(s)\\
& \phantom{=} ~~~~~   + \sum_{i,j=1}^m (\cv^2(s))_{ij}\nabla_{X_i}X_j (\xi^{t,x,\cv(\cdot)}(s)) ds, \quad s\in (t,T]\\
\xi^{t,x,\cv(\cdot)}(t)&=x
\end{aligned}
\right.
\end{equation} where $\nabla_{X_i}X_j = DX_j \cdot X_i$ is the (Euclidean)  derivative of the vector field $X_j$ in direction $X_i$.  
A straightforward application of  It\^o's formula  gives for smooth bounded $u: \R^n \to \R$ that 
\begin{align}
 du(\xi^{t,x,\cv}(s)) & = \sqrt 2 \sum _{i=1}^m X_i(u)(\xi^{t,x,\cv}(s))\cv(s) dB(s)\nonumber \\
& \phantom{~~} + \sum_{{i,j=1}}^{m} (\cv^2(s))_{ij}  \left[\sum_{k,l=1}^n u_{kl}   X_i^k  X_j ^l  + 
\sum_{k=1}^n u_{k}   \nabla_{X_i}X_j^k\right]\left(\xi^{t,x,\cv}(s)\right)  ds,\nonumber 
\intertext{where we used the notation $X_i= (X_i^1, \dots,X_i^n) \in \R^n$, $u_k=  \frac {\partial  u}{\partial x_k}$ and  $u_{kl} = \frac {\partial ^2 u}{\partial x_k \partial x_l}$, so that the previous identity can be written as}
 du(\xi^{t,x,\cv}(s)) 
 &= \sqrt 2 \sum _{i=1}^m X_i(u)(\xi^{t,x,\cv}(s))\cv(s) dB(s) +  \mathop{\normalfont tr}\left[ (\cv(s))^2   (\cX^2u)\right]( \xi^{t,x,\cv}(s) )  ds\nonumber  
 \label{generatorformula} 
 \end{align}

From now on, we assume that the matrix $\sigma(x)$ as well as the drift  $$\mu(x):= \sum_{i,j=1}^m  \nabla_{X_i}X_j (x)$$ are Lipschitz in $x$. Under the Lipschitz condition, classical results for stochastic ODEs give that for any fixed control $\cv$, 
 \eqref{SDEs-Srat} has a unique strong  (see e.g. \cite{Yong}, Chapter 1, Corollary 6.1). 
Recall that  the main difference between the notions of strong and weak solutions is that, in the first case, the filtered probability space  $(\Omega, \mathcal F, \{\mathcal F_t\}_{t\geq 0}, \dP)$  and the Brownian motion $B$ are fixed while a weak solutions mean that there exists a
process on {\em some}  filtered probability space equipped with an adapted Brownian motion which satisfies the equation,
for more details see Definitions 6.2 and 6.5, \cite{Yong}. 
This difference becomes very important for the stochastic control problem, i.e. when considering
\begin{equation}
\label{carboni}
 \inf_{\cv}{ \mathbb E}[f\big(\xi^{t,x,\cv}(T)\big)],
\end{equation}
  where usually $f$ is a suitably regular teminal cost  function and $\xi^{t,x,\cv}(\cdot)$ are solutions of a controlled  It\^o SDE as e.g.  \eqref{SDEs-Ito}.
  It is clear that the properties of the previous minimum problem depend on the set where we take the infimum.\
 Hence we define the set $\cA$ of all the weak-admissible controlled pairs (\cite{Yong}, Definition 4.2) which are, roughly speaking, 6-tuple $\pi=(\Omega, \mathcal F, \{\mathcal F_t\}_{t\geq 0}, \dP, B(\cdot), \cv(\cdot))$
  such that $( \xi^{t,x,\cv}(\cdot), (\Omega, \mathcal F, \{\mathcal F_t\}_{t\geq 0}, \dP)) $ is a weak solution of the controlled SDE  \eqref{SDEs-Srat}, 
 w.r.t.  the control $\cv$ and the Brownian motion $B$ in the filtered probability space $(\Omega, \mathcal F, \{\mathcal F_t\}_{t\geq 0})$. 
Under certain structural assumptions for the control set and assuming sufficient regularity of the coefficients, 
the existence of an optimal control for a large class of problems as in \eqref{carboni} is known,
(see for example Theorem 5.3 in \cite{Yong}). For these results it is crucial to use the weak formulation.\\

 Following \cite{soner3,soner1,buckdahn}, for a given bounded and uniformly continuous function $g:\R^n\to \R$, we define   the value function associated to the stochastic control probems \eqref{SDEs-Srat}, as
 \begin{equation}
\label{reachSet}
V(t,x)=\inf_{\cv \in \cA}\ess\sup_{ \omega\in\Omega}g(\xi^{t,x,\cv(\cdot)}(T)(\omega))
\end{equation}
 where the set $\cA$ is the set of the weak-admissible controlled pairs, defined above.\\
 From now to on we will often omit the dependency on $\omega$.\\
 
 In the Euclidean case (i.e. $\sigma(x)=Id$) the value function \eqref{reachSet}  is  the solution  of the level set equation for the evolution by mean curvature flow (backward)  in the viscosity sense (cf.  \cite{buckdahn}, Theorem 1.1),  
\begin{equation}
\label{BackMChor}
\left\{
\begin{aligned}
-V_t=\Delta V-\Delta_{\infty} V, \quad &x\times \R^n,\,t\in [0,T)\\
V(T,x)=g(x),\quad &x\in \R^n.
\end{aligned}
\right.
\end{equation}
Our goal is to show that this result is still true in the general sub-Riemannian case.\\
 We conclude this section which was devoted to stochastic control problems  with a remark. 
 
 \begin{rem} [Non-Lipschitz coefficients]
 If the coefficients of the matrix $\sigma(x)$ and of the drift part $\mu(x)$ are smooth but not globally Lipschitz,
 the solutions of the SDE could explode in finite time. 
 There are results on non-explosion for some classes of non-Lipschitz coefficients, but we will not investigate this issue further,
 but instead assume global in time existence of solutions for the controlled SDEs. In many important
 examples, e.g. in the Heisenberg group, the Lipschitz condition is satisfied and so the non-explosion follows.
 In particular, in this case the drift part is zero, so the Stratonovich and the It\^o formulations coincide.
  \end{rem}

\subsection{Properties of the value function.}$\quad$\\

{In this section we study the main properties of the value function as defined by \eqref{reachSet}.}
\begin{lemma}[Comparison Principle]\label{comstoch}
Let $g_1, g_2$ be bounded and uniformly continuous functions with $g_1\le g_2$ on $[0,T]\times \R^n,$ and let 
$V_i,$ $i\in \{1,2\},$  be defined as in (\ref{reachSet}) with $g_i $ as terminal cost. Then
$$
V_1(x,t)\le V_2(x,t), \quad \textrm{on}\;  [0,T]\times \R^n
$$
\end{lemma}
The proof is obvious and therefore omitted.

\begin{lemma}[Value function is geometric]\label{geomstoch}
Let $g$ be bounded and uniformly continuous, and let $V_g$ be defined as in (\ref{reachSet}) with $g$ as terminal cost.
Let $\varphi: \R\to\R$ be continuous and strictly increasing. Then
$$
\varphi(V_g(t,x))=V_{\varphi(g)}(t,x)
$$
\end{lemma}
\begin{proof}
 As $\varphi$ is increasing and continuous, $\varphi(\inf A)=\inf \varphi(A)$ for any bounded set $A\subseteq \R.$ 
Hence, for any measurable function $f:\Omega\to \R$, it is trivial to note:
$$ \varphi(\ess\sup f)=\ess\sup(\varphi(f))$$
 and so we can easily conclude the proof. 
 \end{proof}
\begin{rem}
Lemmas \ref{comstoch} and \ref{geomstoch} allow  us to conclude, reasoning as in \cite{Gigabook}, that the sublevel set $\{V(t,x)\le 0\}$ depends only on the set
$\{g(x)\le 0\},$ and not on the specific form of $g$. Therefore we could  introduce a new weak notion for the evolution by horizontal mean curvature by 
defining the zero-level set of the value function $V(t,x)$ as generalized solution at time $t,$ which is well-defined. 
Moreover, in  Theorem \ref{TEOREMA1},  we will show that  $V(t,x)$ solves (in the viscosity sense) the level set equation for the evolution by horizontal mean curvature flow. Unfortunately,   without comparison principles and the resulting uniqueness for solutions of
the PDE, we  cannot say that this new definition is consistent with the classical evolution for regular hypersurfaces.
\end{rem}
\begin{lemma}[Boundedness]\label{bounded}
Assume that  $g$ is bounded and that there exists a weak solution of the  controlled SDE \eqref{SDEs-Srat}  for at least one control $\cv(s)$, then  the value function  $V(t,x)$ defined  in (\ref{reachSet})
is bounded.
\end{lemma}
\begin{proof}
The property follows immediately once we know that the infimum is taken over a non-empty set.
\end{proof}
 In order to investigate the continuity of the value function, we have to restrict our attention to the case of Carnot groups.
\begin{lemma}[Continuity in space]\label{stochcon}
Let $\G=(\R^n,\cdot)$ be a Carnot group, and suppose $g:\G\to \R$ is bounded and uniformly continuous on the one-point-compactification
of $\G$, i.e. it is  uniformly continuous on $\G$ and there exists $\lim_{|x|\to \infty}g(x)$.\\
 Then $V(t,x)$ defined  in (\ref{reachSet})
is continuous in space.
\end{lemma}
\begin{proof} Denote by $L_a(\cdot)$ the left translation in the Carnot group by the element $a\in \G.$ . As $\G$ is a Lie group, we may assume that the vector fields $X_1,\ldots,X_m$ are
left-invariant, i.e. 
 \begin{equation}\label{leftinvariant}
X_i(a\cdot x)=X_i (L_a (x))=(DL_a)(X_i(x)),
\end{equation}
for $i=1,...,m$, where $DL_a$ is the derivative of the left
translation (see e.g. \cite{Lie} for more details on Lie groups).
Let $\xi^{t,x,\cv(\cdot)}$ be a constrained codimension one
horizontal Brownian motion, with
$d\xi^{t,x,\cv(\cdot)}=\sqrt{2}\sigma^T(\xi^{t,x,\cv(\cdot)}(s))
\circ d B^{\cv}(s),$
 then, by the chain rule for Stratonovich integrals, it holds
\begin{multline}\label{translate}
d\left(L_a\left(\xi^{t,x,\cv(\cdot)}\right)\right)=
(DL_a)\circ\left(\sqrt{2}\sigma^T\left(\xi^{t,x,\cv(\cdot)}(s)\right)\circ  d
  B^{\cv}(s)
\right)\\
{=}
\sqrt{2}(DL_a)\left(\sigma^T\left(\xi^{t,x,\cv(\cdot)}(s)\right)\right)
\circ  d B^{\cv}(s)
=\sqrt{2}\sigma^T\left(L_a\left(\xi^{t,x,\cv(\cdot)}(s)
\right)\right)\circ  d B^{\cv}(s)
\end{multline}
where we have used (\ref{leftinvariant}) for the last equality. Hence the left translation of a codimension 1 horizontal Brownian motion
yields another one. 
Now fix a point $x,$  $\epsilon>0$ and choose a control $\cv_x$ such that
$$
V(t,x)+\epsilon \ge \ess\sup g(\xi^{t,x,\cv_x(\cdot)}(T)).
$$
Let $a=y\cdot x^{-1}.$ By (\ref{translate}), the path $\eta^{t,y,\cv_x(\cdot)}$ starting at the time $t$ in $y$, is  equal to $L_a(\xi^{t,x,\cv_x(\cdot)}).$ 
(Note that the control $\cv_x$ is the same for both points $x$ and $y$.)
Therefore
$$
\begin{array}{rcl}
V(t,y)&\le& \ess\sup g(\eta^{t,x,\cv_x(\cdot)}(T))= \ess\sup g(L_a(\xi^{t,x,\cv_x(\cdot)}(T)))
\\&=&\ess\sup \left(g(\xi^{t,x,\cv_x(\cdot)}(T))+\left(g(L_a(\xi^{t,x,\cv_x(\cdot)}(T)))-g(\xi^{t,x,\cv_x(\cdot)}(T)) \right)\right)
\\&\le& V(t,x)+\epsilon
+\ess\sup \left | g\left(L_a(\xi^{t,x,\cv_x(\cdot)}(T))\right)-g\left(\xi^{t,x,\cv_x(\cdot)}(T)\right)\right |
\end{array}
$$
Choose a large number $R>0$ then
\begin{multline*}
\ess\sup_{\omega\in\Omega} \left | g\left(\xi^{t,x,\cv_x(\cdot)}(T)(\omega)\right)-g\left(L_a(\xi^{t,x,\cv_x(\cdot)}(T)(\omega))\right)\right |
\\ \le \sup_{\{z\in \R^n:\  |z|<R\}}|g(z)-g(a\cdot z)|+\sup_{\{z\in \R^n:\  |z|\ge R\}}|g(z)-g(a\cdot z)|=:A+B
\end{multline*}
where we set $z= \xi^{t,x,\cv_x(\cdot)}(T)(\omega)$ and so $a\cdot z=L_a(\xi^{t,x,\cv_x(\cdot)}(T)(\omega))$. \\
Note that $|a\cdot x|\to \infty$ if $|x|\to \infty$.
Therefore we can use the continuity of $g$ at $\infty$ to find a sufficiently large $R$ such that $B<\epsilon.$
As, by continuity of the group operation, $|a\cdot x-x|\to 0$ (uniformly on compact sets) as $|a|=|y\cdot x^{-1}|\to 0,$ 
we can use the uniform continuity of $g$ to find 
$\delta>0$ such that 
$
V(t,y)\le V(t,x)+3\epsilon
$ for $|x-y|<\delta.$ Reversing the role of $x$ and $y$ yields the continuity.
\end{proof}
 \section{Existence of a generalized evolution by horizontal mean curvature flow in  sub-Riemannian manifolds.}

Using the value function for the stochastic control problem introduced in the previous section as representation for the viscosity solution of equation \eqref{levelSetMCE}, we get an existence result for the generalized evolution by horizontal mean curvature flow as given in Definition \ref{giga}. By classical results   (see  e.g. \cite{sonerbook}, \cite{touziNote}), it is known how to find the equation solved by value functions  of the form
$\inf_{\cv\in \cA}\dE\big[g(\xi^{t,x,\cv(\cdot)}(T)\big]$. Unfortunately,
the value function $V(t,x)$ defined in \eqref{reachSet} looks different, because of the essential supremum instead of the expectation. Hence the idea (already used in \cite{buckdahn}) is to approximate formula  \eqref{reachSet}  with functions that look like the infimum of an expectation and then to pass to the limit,
essentially using the fact that the $L_p$-norm of  a fixed nonnegative function converges to the essential supremum as $p\to\infty.$ \\

Since we are only able to show that the value functions defined in  \eqref{reachSet} are  lower semicontinous,  we need to use the viscosity theory for discontinuous functions.
Next we will recall the definition. For more details on this theory, we refer to \cite{barles}.
\begin{defi}
\label{barles}
A  locally bounded function $u:\R^n\times [0,T] \to \R $ is a discontinuous viscosity solution of equation \eqref{BackMChor}, if $u^*(t,x)$  is a viscosity subsolution and $u_*(t,x)$ is a viscosity supersolution of the same equation, 
where $u^*$ and $u_*$
 are respectively  the  upper and lower 
semicontinuous envelope of $u$, i.e.
$$ u^*(t,x):= \inf\{v(t,x) |v \;\textrm{cont. and}\; v\geq u \}= \limsup_{r \to 0^+} \{u(s,y) |\; |y-x|\leq
r, |t-s|\leq r\},
$$
$$ u_*(t,x):= \sup\{ v(t,x) |v \; \textrm{cont. and}\; v\leq u\}=\liminf_{r \to 0^+} \{u(s,y) |\; |y-x|\leq
r|t-s|\leq r\}.
$$
\end{defi}
 The main result of this paper is the following existence theorem.
\begin{teo}
\label{TEOREMA1} 
 Let $g : \R^n \to \R$  bounded and H\"older continuous, $T>0$ and $\sigma(x)=[X_1(x),...,X_m(x)]^T$ a $m\times n$-H\"ormander matrix with $m\leq n$ and smooth coefficients.   Assuming that  $\sigma(x)$ and $\sum_{i=1}^m\nabla_{X_i}X_j(x)$ are Lipschitz (in order to have non-explosion for the solution of the SDE),  then the value function $V(t,x)$ defined by \eqref{reachSet} is a bounded lower semicontinuous  viscosity solution of the level set equation for the evolution by horizontal mean curvature flow \eqref{BackMChor}, with terminal condition $V(T,x)=g(x)$.
\end{teo}
   \begin{rem}[Examples]
We recall that the Lipschitz requirement of the  theorem is always satisfied if the  vector fields $X_1,...,X_m$ 
and their derivatives are smooth with at most linear growth a infinity. 
Hence the previous 
theorem covers many sub-Riemannian geometries, as, for example,  the Heisenberg group, the Gru\v{s}in plane and the roto-translation geometry, 
introduced in Examples \ref{Ex:Heisenberg}, \ref{Ex:Grusin} and, \ref{Ex:Visual}. 
\end{rem}
Next we give a proof that works rigourously just under additional, strong assumptions on the  regularity. In fact, for passing to the   limit, we will need a strong convergence of the approximating functions together with all the derivatives involved. In general those assumptions are not satisfied and indeed $V(t,x)$ is just a viscosity solution and not a classic solution. In the Appendix, we will present a viscosity proof that follows the one used in the Euclidean case by Buckdahn, Cardaliaguet and Quincampoix in \cite{buckdahn}. Nevertheless, the following proof gives a clear idea of why the 
optimal control, realizing the infimum in \eqref{reachSet}, is given by the projection on  the tangent space of the level set of the solution at that point.\\
 Before we start with the heuristic arguments, let us introduce a regularization which will be used both for the heuristics and for a rigorous viscosity
proof.
As the essential supremum is the limit of $L^p$-norms, we consider, following \cite{buckdahn},
for any $1<p<+\infty$, the following value functions:
\begin{equation}
\label{approximatingFunction}
V_p(t,x) = \inf _{\cv\in \cA}\ \left(\dE\big[g^p(\xi^{t,x,\cv(\cdot)}(T))\big]\right)^{1/p}.
\end{equation}
Note that we can replace $g$ by
$ag+b$ for real numbers $a$ and $b$ and therefore, as $g$ is bounded, assume
$C\geq g (x)\geq 1$, (for any $C>0$) \\
The idea is to derive the PDE solved by  the value functions \eqref{approximatingFunction} and then to show that $V$ is  
their limit as $p\to +\infty$ and solves  a limit-equation which is exactly the level set equation for the evolution by horizontal mean curvature flow,
\eqref{levelSetMCE}.  In fact, we have
\begin{lemma}\label{pointconv}
Under the assumptions of Theorem \ref{TEOREMA1}, we have 
\begin{equation}
\label{limiteOlivia}
V(t,x)=\lim_{p\to +\infty} V_p(t,x)\quad {\rm\ for\ all\ } (t,x)\in [0,T]\times \R^N \ (pointwise\ convergence)
\end{equation}
with $V(t,x)$ as in (\ref{reachSet}).
\end{lemma}
As the $L^p$ norms are increasing and bounded by the essential supremum, for each fixed control, it is clear that 
$V(t,x)\ge \lim V_p(t,x).$ In order to show equality, we can argue  as in \cite{buckdahn}, noting that in our case the controlled SDE 
has a drift part, but it depends
on the control only through $\nu^2,$ and our control set is, as the one in \cite{buckdahn}, convex in $\nu^2.$ Hence similar arguments,
based on \cite{Yong}, apply.
Let us now explain heuristically how to find the equation solved by $V.$
\begin{proof} [Heuristic proof]
 We first look at
\begin{equation}
\label{Up}
U_p(t,x) =V^p_p(t,x)= \inf _{v \in \mathcal V} \dE\big[g^p(\xi^{t,v,x}(T))\big]
\end{equation}
It is known (see  e.g. \cite{sonerbook}, \cite{touziNote}) that $U_p(t,x)$ is a viscosity solution of
\begin{equation}
\label{PDE_p^p}
\left\{
\begin{aligned}
-(U_p)_t+H(x,DU_p,D^2U_p)=0,\quad &x\in  \R^n,\,t\in [0,T)\\
U_p(T,x)=g^p(x),\quad &x\in \R^n
\end{aligned}
\right.
\end{equation}
where 
\begin{equation}
\label{formulaTouzi}
H(x,p,S)=\sup_{\cv\in \cA}\left[-\frac{1}{2}\tr\big[F(x,\cv)F^T(x,\cv)S\big] + \big<\mu(x,\cv), p\big>\right]
\end{equation}
whenever $ \xi^{t,v,x}(T)$ is a solution of the stochastic ODE
$$
\left\{
\begin{aligned}
d\xi^{t,x,\cv(\cdot)}(s)&=F(\xi^{t,x,\cv(\cdot)}(s),\cv)  d B(s)+ \mu(\xi^{t,x,\cv(\cdot)}(s),\cv) ds, \quad s\in (t,T]\\
\xi^{t,x,\cv(\cdot)}(t)&=x
\end{aligned}
\right.
$$
In order to compute the structure of the Hamiltonian associated to \eqref{Up}, we have to consider the  It\^o formulation of the stochastic control problem, given in \eqref{SDEs-Ito}, that means $F(x,\cv)=\sqrt{2}\sigma^T(x) \cv(s)  $ and $\mu(x,\cv)=\sum_{i,j=1}^m (\cv^2(s))\nabla_{X_i}X_j(x)$. Hence \eqref{formulaTouzi} becomes
\begin{equation}
\label{eva}
H(x,p,S)=\sup_{\cv\in \cA}\left[-\tr\big(\sigma(x)S\sigma^T(x)\cv^2(s)\big)+\sum_{i,j=1}^m (\cv^2(s))_{i,j}\big<\nabla_{X_i}X_j(x),p\big>
\right]
\end{equation}
for any $x,p\in \R^n$ and any symmetric $n\times n$ matrix $S.$ \\

Unlike for Euclidean  mean curvature flow, our Hamiltonian depends on the gradient and on space, not just on $S$, (In fact in the Euclidean case 
$\nabla_{X_i}X_j=0$ and $\sigma(x)=Id$).
Assuming $g$ locally Lipschitz (or locally H\"older) and $\sigma(x)$ smooth (hence locally Lipschitz), it is possible to show the continuity of $U_p$, exactly as in Propositions 1.3 and 1.4  in \cite{touziNote}. So $U_p$ is a (continuous) viscosity solution of \eqref{PDE_p^p} with $H$ given by \eqref{eva}.\\

Once we know the  equation solved by $U_p(t,x)=V_p^p(t,x)$, we can easily find the corresponding PDE for $V_p$.\\
Since $C\geq g (x)\geq 1,$   $V_p\geq 1>0$, too, and  we can divide by 
$pV_p^{p-1}$. Assuming that all functions involved are smooth, a trivial calculation tells that $V_p$ solves
\begin{equation}
\label{PDE_p}
\left\{
\begin{aligned}
-(V_p)_t+H(x,DV_p,(p-1)V_p^{-1}DV_p(DV_p)^T+D^2V_p)=0,\quad &x\in \R^n,\, t\in [0,T)\\
V_p(T,x)=g(x),\quad &x\in \R^n
\end{aligned}
\right.
\end{equation}
Whenever  $V_p$ is just continuous, we can show that it solves equation \eqref{PDE_p} in the viscosity sense, by applying the previous  calculation  to the (smooth) test functions. 
The continuity for $V_p$  follows from the continuity for $U_p$.\\
The following computations conclude the proof, assuming $V_p$ and $V$ are so regular that   
$$
\big(V_p\big)_t\to V_t,\quad DV_p\to DV,
\quad D^2V_p\to D^2V \quad \textrm{as}\;p\to +\infty
$$
As we already remarked, those assumptions are in general not satisfied and so we will give in the Appendix 
a more rigourous viscosity proof.\\
However,
the following computation explains why
the supremum in \eqref{eva} is attained whenever
$\overline{\cv}=I_m-n_0\otimes n_0$, where $n_0 $ is the horizontal normal to the level
set (at the non-characteristc points).\\
So, let us first assume  
$\cX V(t,x)=\sigma(x)DV(t,x)\neq 0$ (implies $\cX V_p(t,x)\neq 0$, at least, for large $p$).
We can write  explicitely the Hamiltonian in \eqref{PDE_p}:
\begin{multline}
\label{Roberto0}
 H(x,DV_p,(p-1)V_p^{-1}DV_p(DV_p)^T+D^2V_p)=\\
 \sup_{\cv \in \cA}\bigg[-(p-1) \tr\big[V_p^{-1} \cv \cv^T (\sigma(x)DV_p)(\sigma(x)DV_p)^T\big]+\tr\big[\cv\cv^T\sigma(x) D^2 V_p\sigma^T(x)\big]\\
 + \sum_{i,j=1}^m (\cv^2(s))_{i,j}\big<\nabla_{X_i}X_j(x),DV_p\big>\bigg]
\end{multline}
Recalling that $(\cX^2V_p)^*=\sigma(x) D^2V_p \sigma^T(x)+A(x,DV_p)$ where $A(x,p)$ is defined by \eqref{matricial}, we  observe that
$$
\tr(\cX^2V_p)^*=\tr \big(\sigma(x) D^2V_p \sigma^T(x)\big)+\tr \big(A(x,DV_p)\big)
$$
The trace of the first order part is 
$$
\tr A(x,DV_p)= \sum_{i=1}^m\big<\nabla_{X_i}X_i(x),DV_p\big>
$$
so Hamiltonian \eqref{Roberto0} can be rewritten in ``horizontal notation'' as
\begin{multline}
\label{Roberto}
  H(x,DV_p,(p-1)V_p^{-1}DV_p(DV_p)^T+D^2V_p)= \\
  \sup_{\cv \in \cA}\bigg[-(p-1)V_p^{-1} \tr\big[ \cv \cv^T (\cX V_p)(\cX V_p)^T\big]+\tr\big[\cv \cv^T (\cX^2 V_p)^*\big]\bigg]
\end{multline}
Note that 
 $$
-(p-1)V_p^{-1} \tr\big[ \cv \cv^T (\cX V_p)(\cX V_p)^T\big]
=-(p-1)V_p^{-1}\tr\big[\big( \cv^T\cX V_p\big)\big(\cv^T\cX V_p\big)^T\big]
\le 0,
$$
and so it goes to $-\infty$ as $p\to +\infty$.
Hence, in order to attain the supremum, we need (at least for large $p$) that
$\cv^T\cX V_p=0$. 
Since the horizontal gradient is in the direction of the horizontal normal,  the optimal control $\overline{\cv}$ has to  coincide with the projection on the tangent space, that means 
\begin{equation}
\label{optimalControl}
\overline{\cv}=I_m-n_0\otimes n_0
\end{equation}
To get the level set equation, we have to write Hamiltonian \eqref{Roberto0}  in a bit different way. Let $I_m$ be the $m\times m$ identity-matrix,   we can replace
 $\cv^2$ by $I_m-a \otimes a$ with $a\in \R^m$, then, for any $n\times n$ matrix $S$, it holds
 $$
\sup_{\cv \in \cA} \big[-\tr\big(\cv^2S\big)]
=\sup_{\cv \in \cA} \big[-\tr\big((I_m-a \otimes a)S\big)]
=-\tr[S]+\max_{|a|=1}\big<Sa,a\big>$$
Using the optimal control \eqref{optimalControl}  and recalling that $n_0=\cX V/|\cX V|$ and $S = (\cX^2 V)^*$,  we can conclude that  the limit Hamiltonian, as $p\to \infty$, is  
$$
H(x, DV,D^2 V)
=-\tr \big[(\cX^2 V)^*\big]+\bigg<(\cX^2 V)^*\frac{\cX V}{|\cX V|},\frac{\cX V}{|\cX V|}\bigg>=-\Delta_0V+\Delta_{0,\infty} V
$$
So the limit equation of problems   \eqref{PDE_p},  as $p\to +\infty$, is exactly    the level set equation for the horizontal evolution by mean curvature, at  non-characteristic points. \\
Next we would like to investigate what happens at the characteristic points, so we assume $\cX V(t,x)=0$. Passing to the limit  in \eqref{Roberto}, the first order disapears whatever the control $\cv$  looks like. Then we get, for any control $\cv=I_m-a\otimes a$ and $S = (\cX^2 V)^*$, 
\begin{multline*}
0=-V_t+\sup_{\cv\in \cA}\bigg[-\tr\big[\cv\cv^T (\cX^2 V)^*\big]\bigg]=-V_t-\tr (\cX^2 V)^*+\max_{|a|=1}\tr\big(aa^T(\cX^2 V)^*\big)\\
=-V_t-\Delta_0 V+\lambda_{\max}(\cX^2 V)^*
\end{multline*}
So we find, as expected, 
the upper semicontinous regularization of the equation.  \\

We like to conclude this proof with some remarks on the regularity of $V(t,x)$ and the initial condition.\\
First we can note that $V(t,x)$ is bounded since the datum $g$ is. 
Moreover, since  $V_p$ is a 
non-decreasing sequence  of continuous functions, we have 
\begin{equation}
\label{lago}
V(t,x)= \lim_{p\to +\infty} V_p(t,x)=\sup_{p>1} V_p(t,x),
\end{equation}
Hence $V(t,x)$ is, a priori,  just lower semicontinuous. To conclude, we observe that   $V(T,x)=\lim_{p\to +\infty}V_p(T,x)=g(x)$. 
\end{proof}
\begin{ex}
In the case of the Heisenberg group, all  the hypersurfaces quoted in  Examples \ref{variuosEx} are covered by the existence result, given by Theorem \ref{TEOREMA1}.
In fact, whenever the hypersurface is compact and $C^1$ (like in the Euclidean ball), we can choose as initial datum $g$, the signed Euclidean distance in a neighborhood, extended  continuously by constants outside, so that the corresponding $g$  is bounded and Lipschitz continuous. For the Carnot-Carath\'eodory ball, we can define $g(x)=d(x,0)-1$ locally around  the surface and  constant outside. The datum in this case is  bounded and locally H\"older continuous (since it is Lipschitz w.r.t. the Carnot-Carath\'eodory metric $d$).  The same holds for  the Kor\'anyi ball.
\end{ex}

We would like to show  that the functions defined in  \eqref{reachSet} are continuous in time and space and to get so the existence of a continuous viscosity
solution of the level set equation for horizotal mean curvature flow. \\
Unfortunately,  so far we are able to show directly the continuity of the value function   \eqref{reachSet}  just w.r.t. $x$ and for Carnot groups (Lemma \ref{stochcon}).\\
Following the strategy introduced in \cite{buckdahn}, it is possible to get the continuity of the solution via comparison principles for viscosity
subsolutions and supersolutions.
The problem with  generalizing this idea to the sub-Riemannian setting is that comparison principles are known just for special initial data in
Carnot groups, hence so far we are able to   get a full continuous existence result just for the class of initial data covered by Theorem \ref{unicita}
by \cite{capognaCitti}.
\begin{prop}
\label{TEOREMA2} 
 Let $g : \R^n \to \R$  bounded and H\"older continuous, $T>0$ and $\sigma(x)$ an $m\times n$-H\"ormander matrix like in Theorem \ref{TEOREMA1}. If comparison principles for \eqref{BackMChor} hold, then the value function $V(t,x)$ defined by \eqref{reachSet} is the unique continuous viscosity solution of the level set equation \eqref{BackMChor}, satisfying $V(T,x)=g(x)$.
\end{prop}
Let $V^{\sharp}(t,x)$ the  half-relaxed upper-limit of $V_p(t,x)$, i.e.
\begin{equation}
\label{upperV}
V^{\sharp}(t,x)=\limsup_{{\small\begin{aligned}
(s,y)&\to (t,x)\\
p&\to +\infty
\end{aligned}}}V_p(s,y)
\end{equation}
We remark that, in our case $V^{\sharp}(t,x)=V^*(t,x)$, (more details on this point,  will be given in the Appendix). We can in this way prove
the following lemma, which is the key-point in order to apply comparison principles and obtain  the continuity of the solution.
\begin{lemma}
\label{datum}
For any $x\in \R^n$, $V^{\sharp}(T,x)\leq g(x)$.
\end{lemma}
For a proof of the lemma, we refer to the Appendix.
\begin{proof} [Proof of Proposition  \ref{TEOREMA2}]
By Theorem \ref{TEOREMA1} and Lemma \ref{datum}, once we have comparison principles, it is trivial to conclude.  
In the viscosity proof of \ref{TEOREMA1} (given in the Appendix), we will show that  $V^*(t,x)=V^{\sharp}(t,x)$ is a viscosity subsolution while $V_*(t,x)=V(t,x)$ is a viscosity  supersolution of  equation \eqref{BackMChor}. Since, by  Lemma \ref{datum},
$V^{\sharp}(T,x)\leq g(x)$ while $g(x)=V(T,x)$, comparison principles imply $V^{\sharp}(t,x)\leq V(t,x)$. 
Moreover $V^{\sharp}(t,x)\geq V(t,x)$ by definition. Hence  $V^{\sharp}(t,x)= V(t,x)$, which means $V(t,x)$ upper semicontinuous. Since $V(t,x)$ is already lower semicontinuous as supremum of continuous functions, we conclude that $V(t,x)$ is continuous.
\end{proof}
%
The previous conditional result becomes a full result in the cases where comparison principles are known, as   in some particular hypersurfaces in Carnot groups (\cite{capognaCitti}).
\begin{coroll}
In Carnot groups and under the assumptions of Theorem 
\ref{TEOREMA1} and Theorem \ref{unicita},
the value function $V(t,x)$ defined in \eqref{reachSet} is the unique continuous viscosity solution of the level set equation \eqref{BackMChor}, with $V(T,x)=g(x)$.
\end{coroll}
\section{Appendix.}
In this appendix we give  the proofs omitted  in previous section.\\
Next proof uses ideas from \cite{buckdahn}, applied to the horizontal case, to give a formal viscosity proof of the main theorem.
\begin{proof}[Proof of Theorem \ref{TEOREMA1}]
We have to show that $V(t,x)$ solves the horizontal level set equation in the sense of Definition \ref{barles}.\\
First we recall that, since $V(t,x)$ is lower semicontinuous in time and space, then $V_*(t,x)=V(t,x)$. 
Let $V^{\sharp}(t,x)$ the  half-relaxed upper-limit of $V_p(t,x)$, defined in \eqref{upperV}. 
We remark that $V^{\sharp}\geq V$ and $V^{\sharp}$ upper semicountinuous. Since the upper semicontinuous envelope $V^*$ is the smallest upper semicontinuous function above $V$, then $V^*(t,x)\leq V^{\sharp}(t,x)$. Moreover since $V_p(t,x)$ is non-decreasing, $V_p(t,x)\leq V(t,x)$, for any $x$, $t$ and $p>1$. Taking the limsup in $x,t,p$, we get also the reverse inequality $V^{\sharp}(t,x)\leq V^*(t,x)$, hence $V^*=V^{\sharp}$. Therefore, to verify Definition \ref{barles}, we have to show that $V(t,x)$ is a viscosity supersolution and $V^{\sharp}(t,x)$ is a viscosity subsolution.\\

First we show that $V(t,x)$ is a supersolution of  \eqref{BackMChor}.
So let $\varphi\in C^1([0,T];C^2(\R^n))$ such that $V-\varphi$ has a local maximum at $(t,x)$. Two different cases occur. \\
If $\cX\varphi(t,x)\neq 0$, we have to verify that 
\begin{equation}
\label{supersolution1}
-\varphi_t(t,x)-\Delta_0 \varphi(t,x)+\Delta_{0,\infty} \varphi(t,x)\geq 0
\end{equation}
while, if $\cX\varphi(t,x)= 0$, we need to check
\begin{equation}
\label{supersolution2}
-\varphi_t(t,x)-\Delta_0 \varphi(t,x)+\lambda_{\max}\big((\cX^2\varphi)^*(t,x)\big)\geq 0
\end{equation}
Note that, for any $p>1$, there exists $(t_p,x_p)$ such that $V_p-\varphi$ has a local minimum at $(t_p,x_p)$ and $(t_p,x_p)\to (t,x).$ 
In fact, we can always assume that $(t,x)$ is a strict minimum in some $B_R(t,x)$. Set $K=\overline{B_{\frac{R}{2}}(t,x)}$,  the sequence of minimum points $(t_p,x_p)$ converge to some $(\overline{t},\overline{x})\in K$. 
As $V$ is the limit of the 
$V_p$  and   lower semicontinuous,  therefore a standard argument yields that $(\overline{x},\overline{t})$ is a minimum, hence it equals 
$(x,t).$
Since $V_p$ is a solution of \eqref{PDE_p},  then 
$$
-\varphi_t(t_p,x_p)+H(x_p, (p-1)V_p^{-1} D\varphi (D \varphi)^T +D^2\varphi)(t_p,x_p)\geq 0
$$
In the case $\sigma (x) D\varphi(t,x)\neq 0$, we can write the Hamiltonian in a more explicit way. Set
\begin{align*}
&S_1=(p-1) V_p^{-1}\cX\varphi(t_p,x_p)(\cX\varphi(t_p,x_p))^T
\\
&S_2=(\cX^2\varphi)^*(t_p,x_p)
\end{align*}
\begin{multline}
\label{calcolo}
H(x_p, S_1,S_2)=-\tr ( S_1+S_2)+\lambda_{\max}( S_1+S_2)=-\tr(S_1)-\tr(S_2)+\lambda_{\max}( S_1+S_2)\\=
-(p-1)V_p^{-1}(t_p,x_p)|\cX\varphi(t_p,x_p)|^2-\Delta_0 \varphi(t_p,x_p)
+\lambda_{\max}( S_1+S_2)
\end{multline}
since the  trace operator is linear and $\tr((\cX\varphi(x_p))(\cX\varphi)^T(x_p))=|\cX\varphi(x_p)|^2$.\\
Now we need the following result.
\begin{lemma}[\cite{buckdahn}, Lemma 1.2]
Let $S$ a symmetric $m\times m$-matrix  such that the space of the eigenvectors associated to the maximum eigevalue is of dimension one. Then,  $S\to \lambda_{\max}(S)$ is $C^1$ in a neighborhood of $S$. Moreover, $D\lambda_{\max}(S)(H)=\big<Ha,a\big>$, for any $a\in \R^m$ eigenvector associated to $  \lambda_{\max}(S)$ and $|a|=1$.
\end{lemma}
 We apply previous lemma to the matrix
$$S_p=V_p^{-1}(t,x) \cX \varphi(t_p,x_p)( \cX\varphi (t_p,x_p))^T $$
which clearly satisfies the assumptions of previous lemma.\\ 
Expanding the Hamiltonian \eqref{calcolo} around $S_p$  and then, passing to the limit as $p\to+\infty$, we get exactly \eqref{supersolution1}.\\
The case $\cX\varphi(t,x)= 0$
is much easier. We have just to use the subadditivity of the function   $S\to \lambda_{\max}(S)$ and remark that, since $V_p$ is supersolution
\begin{multline*}
0\leq -\varphi_t+H(x_p,DV_p,(p-1)V_p^{-1}D\varphi(D\varphi)^T+D^2\varphi)\\
\leq -\varphi_t-(p-1)V_p^{-1}|\cX \varphi|^2-\tr\big((\cX^2 \varphi)^*\big)
+\lambda_{\max}\big((p-1)V_p^{-1}\cX\varphi(\cX\varphi)^T+(\cX^2 \varphi)^*\big)
\\
\leq -\varphi_t-(p-1)V_p^{-1}|\cX \varphi|^2-\Delta_0\varphi+(p-1)V_p^{-1}|\cX \varphi|^2+\lambda_{\max}((\cX^2 \varphi)^*)
\end{multline*}
at the point $(t_p,x_p)$.
So, passing to the limit as $p\to +\infty$,  we find \eqref{supersolution2}.\\

To verify the subsolution property for $V^{*}=V^{\sharp}$,  
let  $\varphi\in C^1([0,T];C^2(\R^n))$ such that $V^{\sharp}-\varphi $ has a maximum at $(t_0,x_0)$ and we may assume that such a maximum is strict. Let $(t_p,x_p)$ a sequence of maximum points of $V_p-\varphi$, we can find a subsequence converging to $x$. Hence, since
$V_p$ are solutions of \eqref{PDE_p}, we have 
\begin{equation}
\label{fame}
0\leq -\varphi_t+H(x,(p-1)\varphi^{-1} D\varphi (D \varphi)^T +D^2\varphi) 
\end{equation}
at $(t_p,x_p)$.
We define, for any  $z>0$, $x,d\in \R^n$ and any  $n\times n$ symmetric matrix $S.$
\begin{multline}
\label{geronimo}
H_p(x,z, d,S)= -\frac{(p-1)}{z} |\sigma(x)d+A(x,d)|^2-\tr (\sigma^T (x)S\sigma(x)+A(x,d))\\+\lambda_{\max}\bigg( \frac{(p-1)}{z} (\sigma(x) d)(\sigma(x) d)^T+\sigma^T(x) S\sigma(x)+A(x,d)\bigg)
\end{multline}
and
\begin{equation}
\overline{H}(x,d,S)\!\!=\! \!\left\{\!
\begin{aligned}
-\tr (\sigma^T (x) S\sigma(x)&+A(x,d))+\\
&\bigg<\!(\sigma^T(x) S\sigma(x)\!+\!A(x,d))\frac{\sigma(x)d}{|\sigma(x)d|},\frac{\sigma(x)d}{|\sigma(x)d|}\!\bigg>, \, |\sigma(x) d|\neq \!0\\
&-\tr (\sigma^T (x)S\sigma(x))+\lambda_{\max}(\sigma^T (x) S\sigma(x)), \quad |\sigma(x)d|=\!0
\end{aligned}
\right.
\end{equation}
Since $\lambda_{\max}(\sigma^T S\sigma+A)\geq \lambda_{\min}(\sigma^T S\sigma+A)$, it is clear that $$
\overline{H}_*(x,d,S)=H_*(x,d,S)$$
Moreover, as in \cite{buckdahn}, we can observe that
$$
H_p(x,z,d,S)\geq \overline{H}(x,d,S),\quad\textrm{for any}\;z
$$
which is trivial for $|\sigma(x)d|=0$ (by \eqref{geronimo}) and for $|\sigma(x)d|\neq 0$, it follows taking $a=\frac{\sigma(x)d}{|\sigma(x)d|}$ in the definition of maximum eigenvalue as $\lambda_{\max}(\widetilde{S})=\max_{|a|=1}\big<\widetilde{S}a,a\big>$.\\
Set $z=\varphi^{-1}(t_p,x_p)>0$, $d= D \varphi (t_p,x_p)$, $S=D^2 \varphi (t_p,x_p) $, by \eqref{fame} taking the limsup as $p\to +\infty$, we can deduce that
$$
0\geq \varphi_t+ H_*(x, D\varphi, D^2\varphi)
$$
at $(t,x)$.
That concludes the proof since it shows that the value function $V(t,x)$ satisfies Definition \ref{barles}.
\end{proof}
We conclude by giving a proof of  Lemma \ref {datum}, the key point for obtaining the continuity of the functions defined by the 
stochastic representation formula via the comparison principle.
\begin{proof} [Proof of Lemma \ref {datum}]
Assuming that it is not true, there exists a point $x_0$ such that 
$$V^{\sharp}(T,x_0) \geq g(x_0)+\varepsilon$$
for $\varepsilon>0$ sufficiently small. 
Then we use as test function
$$
\varphi(t,x)=\alpha(T-t)+\beta |x-x_0|^2 +g(x_0)+\frac{\varepsilon}{2}
$$
with $\alpha>-C \beta $, with $C$ a constant depending just on the data of the problem and the point $x_0$ (in the Euclidean case $C=-2(n-1)$) and $\beta>1$ sufficiently large. 
Now we can find  a sequence $(t_x,x_k) \to (T,x_0)$ and $p_k \to +\infty$ as $k\to +\infty$ such that $V_{p_k} -\varphi$ has a positive local maximum at some point $(s_k,y_k)$, for any $k>1$ (see \cite{buckdahn} for more details). To get the contradiction we need to use the fact that $V_p$ is a solution (so in particular a subsolution) of  equation \eqref{PDE_p} and, by the subsolution condition, get 
$$
\alpha + C \beta\leq 0
$$
which contradicts the choice $\alpha> -C\beta$.\\
Unfortunately in our case, unlike in the Euclidean case, the test function, inserted in the equation for $V_p$, does not give a constant number since the Hamiltonian  depends on the space-variable. 
Nevertheless, we can observe that the functions $V_p$ are bounded uniformly in $p$ so, by the growth of $|x-x_0|,$ the maximum points  are such that $y_k \in \overline{B_R(x_0)}=:K$, with $R$ independent of $k$. \\
Using this remark and  the continuity  of the terms which we get calculating $H$ in $\varphi$ and at the point $(s_k,y_k)$, 
we will be able to show the following lower bound
$$
\alpha+C\beta \leq -\varphi_t+H((p-1)\varphi^{-1} D\varphi (D \varphi)^T +D^2\varphi)\leq 0
$$
and so conclude.\\
First we need to remark that
$$
\varphi_t(t,x)=-\alpha,\quad D\varphi(t,x)=2\beta |x-x_0|,\quad
D^2\varphi(t,x)=Id
$$
Remarking that  at the point $(s_k,y_k)$, we have
\begin{multline*}
0\geq (p-1)\varphi^{-1} D\varphi (D \varphi)^T +D^2\varphi)\geq
\alpha -\tr\big(\sigma(y_k)\sigma^T(y_x)+A(x,x-x_0)\big)\\
+\lambda_{\min}\big(\sigma(x)\sigma^T(x)+(A(x,x-x_0)\big)
\end{multline*}
Recalling that there is a compact set $K$ such that  $y_k\in K$ for all $k,$ 
we get by continuity
$$
0\geq (p-1)\varphi^{-1} D\varphi (D \varphi)^T +D^2\varphi)\geq \alpha +2C\beta
$$
with 
{\small$$
C\!=\!-\max_{x\in K} \tr(\sigma(x)\sigma^T(x))-\!\max_{x\in K}A(x,x-x_0)+\min_{x\in K}\! \lambda_{\min}(\sigma(x)\sigma^T(x))+\min_{x\in K}\lambda_{\min}(A(x,x-x_0)) 
$$}
With such an estimate, we are able to obtain the same contradiction as in the Euclidean case,  choosing $\alpha> -C\beta$.
\end{proof}


\end{document}